%% file: main.tex
\documentclass[final]{siamltex}

\usepackage{epsfig,amssymb,latexsym}
\usepackage{amsfonts,psfrag,amsmath,bbm,color}
\usepackage{cancel}
\usepackage{comment}
\usepackage{mathrsfs}
\usepackage{graphicx}
\usepackage{textcomp}
\usepackage{multirow}
\usepackage{enumerate}
\usepackage{cancel}
\usepackage{algpseudocode}
\usepackage{caption}  
\usepackage{subcaption}
\usepackage{url}
\usepackage{rotating}
\usepackage{slashbox}
\usepackage{bigints}
\usepackage{hyperref}
\usepackage{xcolor}
\usepackage{ulem}

\usepackage{placeins}

\newcommand{\bchi}{\pmb{\chi}}
\def\grad{{\nabla}}
\usepackage[ruled,vlined]{algorithm2e}
\usepackage{geometry}
\vbadness=\maxdimen
  
\newcommand{\bif}{\textbf{\textit{f}}}
\newcommand{\bg}{\textbf{\textit{g}}}

\newcommand{\footremember}[2]{%
    \footnote{#2}
    \newcounter{#1}
    \setcounter{#1}{\value{footnote}}%
}
\newcommand{\footrecall}[1]{%
    \footnotemark[\value{#1}]%
} 

\usepackage{tikz}
\usetikzlibrary{shapes.geometric, arrows}
\tikzstyle{startstop} = [rectangle, rounded corners, minimum width=1cm, minimum height=1cm,text centered, draw=black]
\tikzstyle{io} = [trapezium, trapezium left angle=70, trapezium right angle=110, minimum width=1cm, minimum height=1cm, text centered, draw=black, fill=blue!30]
\tikzstyle{method} = [rectangle, rounded corners, minimum width=1cm, minimum height =1cm, text centered, draw=black]
\tikzstyle{process} = [rectangle, minimum width=1cm, minimum height=1cm, text centered, draw=black]
\tikzstyle{decision} = [diamond, minimum width=0.5cm, minimum height=0.5cm, text centered, draw=black, fill=green!30]
\tikzstyle{arrow} = [thick,->,>=stealth]

\usepackage{amscd}

\graphicspath{{./figs/}}

\input{notation-siam}

\begin{document}
\title{An efficient algorithm for simulating ensembles of parameterized MHD flow problems}

\author{
M. Mohebujjaman\footremember{mit}{D\MakeLowercase{epartment of} M\MakeLowercase{athematics and} P\MakeLowercase{hysics}, T\MakeLowercase{exas} A\&M I\MakeLowercase{nternational} U\MakeLowercase{niversity}, TX 78041, USA;}\footnote{C\MakeLowercase{orrespondence: m.mohebujjaman@tamiu.edu}}%
\and H. Wang\footrecall{mit}
\and L. Rebholz \footremember{clemson}{S\MakeLowercase{chool of} M\MakeLowercase{athematical and} S\MakeLowercase{tatistical} S\MakeLowercase{ciences}, C\MakeLowercase{lemson} U\MakeLowercase{niversity}, C\MakeLowercase{lemson}, SC 29634, USA;}
\and M. A. A. Mahbub\footremember{comilla}{D\MakeLowercase{epartment of} M\MakeLowercase{athematics}, C\MakeLowercase{omilla} U\MakeLowercase{niversity}, C\MakeLowercase{umilla} 3506, B\MakeLowercase{angladesh;}}
 }

\maketitle

\begin{abstract}
 In this paper, we propose, analyze, and test an efficient algorithm for computing ensemble average of incompressible magnetohydrodynamics (MHD) flows, where instances/members correspond to varying kinematic viscosity, magnetic diffusivity, body forces, and initial conditions. The algorithm is decoupled in Els\"asser variables and permits a shared coefficient matrix for all members at each time-step. Thus, the algorithm is much more computationally efficient than separately computing simulations for each member using usual MHD algorithms. We prove the proposed algorithm is unconditionally stable and convergent. Several numerical tests are given to support the predicted convergence rates. Finally, we test the proposed scheme and observe how the physical behavior changes as the coupling number increases in a lid-driven cavity problem with mean Reynolds number $Re\approx 15000$, and as the deviation of uncertainties in the initial and boundary conditions increases in a channel flow past a step problem.
\end{abstract}



\pagestyle{myheadings}
\thispagestyle{plain}

\markboth{\MakeUppercase{An efficient algorithm for simulating ensembles of parameterized MHD Flow}}{\MakeUppercase{ M. Mohebujjaman, H. Wang, L. Rebholz and M. A. A. Mahbub}}

\section{Introduction}
In this work, we consider the following set of $J$ time-dependent, viscoresistive and incompressible dimensionless magnetohydrodynamics (MHD) equations \cite{B03, D01, LL60, MR17} for computing a MHD flow ensemble simulation of homogeneous Newtonian fluids:
\begin{eqnarray}
\bu_{j,t}+\bu_j\cdot\nabla \bu_j-s\bB_j\cdot\nabla \bB_j-\nu_j \Delta \bu_j+\nabla p_j &= & \bif_j(\bx,t), \hspace{2mm}\text{in}\hspace{2mm}\Omega \times (0,T], \label{gov1}\\
\bB_{j,t}+\bu_j\cdot\nabla \bB_j-\bB_j\cdot\nabla \bu_j-\nu_{m,j}\Delta \bB_j+\nabla\lambda_j &= & \nabla\times \bg_j(\bx,t),\hspace{2mm}\text{in}\hspace{2mm}\Omega \times (0,T],\label{gov3}\\
\nabla\cdot \bu_j & =& 0, \hspace{2mm}\text{in}\hspace{2mm}\Omega \times (0,T], \\
\nabla\cdot \bB_j &=& 0, \hspace{2mm}\text{in}\hspace{2mm}\Omega \times (0,T],\\ 
\bu_j(\bx,0)& =& \bu_j^0(\bx),\hspace{2mm}\text{in}\hspace{2mm}\Omega,\label{gov2}\\
\bB_j(\bx,0)& =& \bB_j^0(\bx),\hspace{2mm}\text{in}\hspace{2mm}\Omega,\label{gov5}
\end{eqnarray}
where $\bu_j$, $\bB_j$, $p_j$, and $\lambda_j$ denote the velocity, magnetic field, pressure, and artificial magnetic pressure solutions, respectively, for each $j=1,2,\cdots\hspace{-0.35mm},J$, corresponding to distinct combination of kinematic viscosity $\nu_j$, magnetic diffusivity $\nu_{m,j}$, body force $\bif_j$, $\nabla\times\bg_j$, and initial conditions $\bu_j^0$, $\bB_j^0$. The symbol $\Omega$ denotes the simulation domain (which we assume to be convex), $t$ the time variable, $\bx$ the spatial variable and $T$ the simulation time. The coupling number $s$ is the coefficient of the Lorentz force into the momentum equation \eqref{gov1}. For simplicity of our analysis, we consider homogeneous Dirichlet boundary conditions.

Input data, e.g., initial and boundary conditions, viscosities, and body forces have a significant effect on simulations of complex dynamical systems, but the involvement of uncertainty in their measurements reduces the accuracy of final solutions. For a robust and high fidelity solution, computation of ensemble average solution is popular in many applications such as surface data assimilation \cite{fujita2007surface}, magnetohydrodynamics \cite{jiang2018efficient}, porous media flow \cite{jiang2021artificial}, weather forecasting \cite{L05,LP08}, spectral methods \cite{LK10}, sensitivity analyses \cite{MX06}, and hydrology \cite{GG11}. Computing a quantity of interest by running a simulation subject to the ensemble average of a particular input data is not always the same as computing the ensemble average of the quantity of interest running the simulations for all different realizations of the input data first and then taking their average \cite{gunzburger2019efficient}.

Computing long-time simulations of a fully coupled MHD ensemble systems is computationally arduous and expensive. Therefore, decoupled algorithms which can reuse the global system matrix at each time-step for all $J$ realizations are computationally attractive. First-order time-stepping partitioned algorithms with small time-step restrictions are studied at low magnetic Reynolds number in a reduced MHD system in \cite{jiang2018efficient}. Decoupled, and unconditionally stable algorithm for the evolutionary full MHD ensemble system in Els\" asser variables are investigated in \cite{MR17}.

Viscosity parameters are the most important and sensitive input data, as they determine the flow characteristics. For example, as the Reynolds number $Re:=UL/\nu$ grows, the laminar flow moves into a convective dominated regime and eventually becomes turbulent \cite{zhang2017critical}. The situation is more complex in MHD flow with high magnetic Reynolds number $Re_m:=UL/\nu_{m}$.  \textcolor{black}{Here}, the contribution of the nonlinearity \textcolor{black}{dominates the flow's development and evolution.} Thus, for an accurate simulation, it \textcolor{black}{is} important to \textcolor{black}{accurately account for their}  uncertainties. The above mentioned MHD ensemble works \cite{jiang2018efficient, MR17} were done assuming uncertainties only on the initial and boundary conditions, and forcing functions; no uncertainties are considered on the viscosity coefficients. In this paper, we propose an algorithm for the MHD flow ensemble in which not only the initial and boundary data, and forcing functions, but also the kinematic viscosity and magnetic diffusivity parameters are different from one ensemble member to another.

Recent studies show \textcolor{black}{that}  instead of solving coupled MHD systems in \textcolor{black}{primitive} variables, using \textcolor{black}{instead} Els\"asser variables \textcolor{black}{can provide} a decoupled stable MHD simulation algorithm, \cite{AKMR15, HMR17, Mohebujjaman2021High, MR17,T14,wilson2015high}. Defining $\bv_j:=\bu_j+\sqrt{s} \bB_j$, $\bw_j:=\bu_j-\sqrt{s}\bB_j$, $\bif_{1,j}:=\bif_{j}+\sqrt{s}\nabla\times \bg_j$, $\bif_{2,j}:=\bif_j-\sqrt{s}\nabla\times \bg_j$, $q_j:=p_j+\sqrt{s}\lambda_j$ and $r_j:=p_j-\sqrt{s}\lambda_j$ produces the Els{\"{a}}sser variable formulation of the ensemble systems:
\begin{eqnarray}
\bv_{j,t}+\bw_j\cdot\nabla \bv_j-\frac{\nu_j+\nu_{m,j}}{2}\Delta \bv_j-\frac{\nu_j-\nu_{m,j}}{2}\Delta \bw_j+\nabla q_j=\bif_{1,j},\label{els1}\\
\bw_{j,t}+\bv_j\cdot\nabla \bw_j-\frac{\nu_j+\nu_{m,j}}{2}\Delta \bw_j-\frac{\nu_j-\nu_{m,j}}{2}\Delta \bv_j+\nabla r_j=\bif_{2,j},\label{els2}\\
\nabla\cdot \bv_j=\nabla\cdot \bw_j=0,\label{els3}
\end{eqnarray}
together with the initial and boundary conditions.

To reduce the immense computational cost for the above ensemble system, we propose a decoupled scheme together with the breakthrough idea \textcolor{black}{Jiang and Layton from} \cite{JL14}. Thus, we consider a uniform time-step size $\Delta t$ and let $t_n=n\Delta t$ for $n=0, 1, \cdots$, (suppress the spatial discretization momentarily), then computing the $J$ solutions independently, takes the following form:\\ 
Step 1: For $j=1,\cdots\hspace{-0.35mm},J$, 
\begin{align}
\frac{\bv_j^{n+1}}{\Delta t}+<\bw>^n\cdot\nabla \bv_j^{n+1}&-\frac{\Bar{\nu}+\Bar{\nu}_{m}}{2}\Delta \bv_j^{n+1}-\nabla\cdot\left(2\nu_T(\bw^{'},t^n)\nabla \bv_j^{n+1}\right)  +\nabla q_j^{n+1} \nonumber\\&= \bif_{1,j}(t^{n+1})+\frac{\bv_j^n}{\Delta t}-\bw_j^{'n}\cdot\nabla \bv_j^n+\frac{\nu_j^{'}+\nu_{m,j}^{'}}{2}\Delta \bv_j^{n}+\frac{\nu_j-\nu_{m,j}}{2}\Delta \bw_j^n,\label{scheme1}\\\nabla\cdot \bv_j^{n+1}&=0.\label{incom1}
\end{align}
Step 2: For $j=1,\cdots\hspace{-0.35mm},J$,
\begin{align}
\frac{\bw_j^{n+1}}{\Delta t}+<\bv>^n\cdot\nabla \bw_j^{n+1}&-\frac{\Bar{\nu}+\Bar{\nu}_{m}}{2}\Delta \bw_j^{n+1}-\nabla\cdot\left(2\nu_T(\bv^{'},t^n)\nabla \bw_j^{n+1}\right)  +\nabla r_j^{n+1} \nonumber\\&= \bif_{2,j}(t^{n+1})+\frac{\bw_j^n}{\Delta t}-\bv_j^{'n}\cdot\nabla \bw_j^n+\frac{\nu_j^{'}+\nu_{m,j}^{'}}{2}\Delta \bw_j^{n}+\frac{\nu_j-\nu_{m,j}}{2}\Delta \bv_j^n,\label{scheme2}\\\nabla\cdot \bw_j^{n+1}&=0.\label{incom2}
\end{align}
\textcolor{black}{Here,} $\bv_j^n,\bw_j^{n}, q_j^n$, and $r_j^{n}$ denote approximations of $\bv_j(\cdot,t^n),\bw_j(\cdot,t^n), q_j(\cdot,t^n)$, and $r_j(\cdot,t^n)$, respectively. The ensemble mean and fluctuation about the mean are defined as follows:
\begin{align*}
<\bz>^n:&=\frac{1}{J}\sum\limits_{j=1}^{J}\bz_j^n, \hspace{2mm} \bz_j^{'n}:=\bz_j^n-<\bz>^n,\\ \Bar{\nu}:&=\frac{1}{J}\sum\limits_{j=1}^{J}\nu_j,\hspace{4mm}\nu_j^{'}:=\nu_j-\Bar{\nu},\hspace{1mm}\text{and}\\ \Bar{\nu}_m:&=\frac{1}{J}\sum\limits_{j=1}^{J}\nu_{j,m},\hspace{2mm}\nu_{j,m}^{'}:=\nu_{j,m}-\Bar{\nu}_m.\label{ensemble_def}    
\end{align*}
The eddy viscosity term, which is $O(\Delta t)$, is defined using mixing length phenomenology, following \cite{JL15},  and is given by \textcolor{black}{$\nu_T(\bz^{'},t^n):=\mu\Delta t(l_z^{n})^2,$
where $\mu$ is a tuning parameter, $l_z^{n}=\max_j|\bz_j^{'n}|$ is a scalar quantity, and $|\cdot|$ denotes length of a vector.}

At each time-step, the above identical subproblems can be solved \textcolor{black}{\textit{simultaneously}} and \textcolor{black}{they each share the exact same} system matrix (which is independent of $j$). \textcolor{black}{Hence, to solve for the next time-step, one} solves the following system of equations \textcolor{black}{of the form} $A[\bx_1|\bx_2|\cdots|\bx_J]=[\bb_1|\bb_2|\cdots|\bb_J]$. Therefore, a massive \textcolor{black}{amount of} computer memory is saved and system matrix assembly \textcolor{black}{and} factorization\textcolor{black}{/preconditioner are needed} only once per time-step. Moreover, the algorithm can take advantage of block linear solvers \textcolor{black}{\cite{ju2020numerical}}. This idea in \cite{JL14} has been implemented for the solution of the heat equation with uncertain temperature-dependent conductivity \cite{Unconditionally2021Fiordilino}, Navier-Stokes simulations \cite{J15,jiang2017second,JL15,NTW16}, magnetohydrodynamics \cite{jiang2018efficient,MR17}, parameterized flow problems \cite{Gunzburger2019Second-Order, GJW17}, and turbulence modeling \cite{JKL15}. Using a finite element spatial discretization, we investigate the \textcolor{black}{proposed} decoupled ensemble scheme \textcolor{black}{\eqref{scheme1}-\eqref{incom2}} in a fully discrete setting. The efficient ensemble scheme is stable and convergent without any time-step restriction, \textcolor{black}{and handles,} uncertainties in all input data.  The \textcolor{black}{rest of the} paper is organized as follows: To follow a smooth analysis, we provide necessary notations and mathematical preliminaries in Section \ref{notation-prelims}. In Section \ref{fully-discrete-scheme}, we present and analyze a fully discrete and decoupled algorithm corresponding to \eqref{scheme1}-\eqref{incom2}, and prove stability and convergent theorems \textcolor{black}{for it}. To support the theoretical analysis, we compute the convergence rates, \textcolor{black}{check the} energy \textcolor{black}{stability} of the scheme, and test the scheme on benchmark problems in Section \ref{numerical-experiments}. Finally, conclusions and future research avenues are given in Section \ref{conclusion-future-works}.

\section{\large Notation and preliminaries}\label{notation-prelims}

Let $\Omega\subset \mathbb{R}^d\ (d=2,3)$ be a convex polygonal or polyhedral domain with boundary $\partial\Omega$. The usual $L^2(\Omega)$ norm and inner product are denoted by $\|.\|$ and $(.,.)$, respectively. Similarly, the $L^p(\Omega)$ norms and the Sobolev $W_p^k(\Omega)$ norms are $\|.\|_{L^p}$ and $\|.\|_{W_p^k}$, respectively for $k\in\mathbb{N},\hspace{1mm}1\le p\le \infty$. The Sobolev space $W_2^k(\Omega)$ is represented by $H^k(\Omega)$ with norm $\|.\|_k$. The vector-valued spaces are $$\bL^p(\Omega)=(L^p(\Omega))^d, \hspace{1mm}\text{and}\hspace{1mm}\bH^k(\Omega)=(H^k(\Omega))^d.$$
For $\bX$ being a normed function space in $\Omega$, $L^p(0,T;\bX)$ is the space of all functions defined on $(0,T]\times\Omega$ for which the following norm
\begin{align*}
    \|\bu\|_{L^p(0,T;\bX)}=\lp\int_0^T\|\bu\|_{\bX}^pdt\rp^\frac{1}{p},\hspace{2mm}p\in[1,\infty)
\end{align*}
is finite. For $p=\infty$, the usual modification is used in the definition of this space. The natural function spaces for our problem are
\begin{align*}
    \bX:&=\bH_0^1(\Omega)=\{\bv\in \bL^p(\Omega) :\nabla \bv\in L^2(\Omega)^{d\times d}, \bv=\textbf{0} \hspace{2mm} \mbox{on}\hspace{2mm}   \partial \Omega\},\\
    Q:&=L_0^2(\Omega)=\{ q\in L^2(\Omega): \int_\Omega q\hspace{1mm}d\bx=0\}.
\end{align*}
Recall the Poincar\'e inequality holds in $\bX$: There exists $C$ depending only on $\Omega$ satisfying for all $\bphi\in \bX$,
\[
\| \bphi \| \le C \| \nabla \bphi \|.
\]
The divergence-free velocity space is given by
$$\bV:=\{\bv\in \bX:(\nabla\cdot \bv, q)=0, \forall q\in Q\}.$$
We define the trilinear form $b:\bX\times \bX\times \bX\rightarrow \mathbb{R}$ by
 \[
 b(\bu,\bv,\bw):=(\bu\cdot\nabla \bv,\bw),
 \]
and recall from \cite{GR86} that $b(\bu,\bv,\bv)=0$ if $\bu\in \bV$, and 
\begin{align}
|b(\bu,\bv,\bw)|\leq C(\Omega)\|\nabla \bu\|\|\nabla \bv\|\|\nabla \bw\|,\hspace{2mm}\mbox{for any}\hspace{2mm}\bu,\bv,\bw\in \bX.\label{nonlinearbound}
\end{align}

The conforming finite element spaces are denoted by $\bX_h\subset \bX$ and  $Q_h\subset Q$, and we assume a regular triangulation $\tau_h(\Omega)$, where $h$ is the maximum triangle diameter.   We assume that $(\bX_h,Q_h)$ satisfies the usual discrete \textit{inf-sup} condition
\begin{align}
\inf_{q_h\in Q_h}\sup_{\bv_h\in \bX_h}\frac{(q_h,\grad\cdot \bv_h)}{\|q_h\|\|\grad \bv_h\|}\geq\beta>0,\label{infsup}
\end{align}
where $\beta$ is independent of $h$. The space of discretely divergence-free functions is defined as 
\begin{align*}
    \bV_h:=\{\bv_h\in \bX_h:(\nabla\cdot \bv_h,q_h)=0,\hspace{2mm}\forall q_h\in Q_h\}.
\end{align*}
For simplicity of our analysis, we will use \textcolor{black}{the} Scott-Vogelius (SV) finite element pair $(\bX_h, Q_h)=\lp(P_k)^d, P_{k-1}^{disc}\rp$,  which satisfies the \textit{inf-sup} condition \textcolor{black}{under certain conditions, such as} when the mesh is created as a barycenter refinement of a regular mesh and the polynomial degree $k\ge d$  \cite{arnold1992quadratic,Z05}. Our analysis can be extended without difficulty to any \textit{inf-sup} stable element choice, \textcolor{black}{although with minor additional technical detail.}

We have the following approximation properties in $(\bX_h,Q_h)$: \cite{BS08}
\begin{align}
\inf_{\bv_h\in \bX_h}\|\bu-\bv_h\|&\leq Ch^{k+1}|\bu|_{k+1},\hspace{2mm}\bu\in \bH^{k+1}(\Omega),\label{AppPro1}\\
 \inf_{\bv_h\in \bX_h}\|\grad (\bu-\bv_h)\|&\leq Ch^{k}|\bu|_{k+1},\hspace{5mm}\bu\in \bH^{k+1}(\Omega),\label{AppPro2}\\
\inf_{q_h\in Q_h}\|p-q_h\|&\leq Ch^k|p|_k,\hspace{10mm}p\in H^k(\Omega),
\end{align}
where $|\cdot|_r$ denotes the $H^r$ or $\bH^r$ seminorm.

We will assume the mesh is sufficiently regular for the inverse inequality to hold, and with this and the LBB assumption, we have approximation properties
\begin{align}
\| \nabla (\bu- P^{L^2}_{\bV_h}(\bu)  ) \|&\leq Ch^{k}|\bu|_{k+1},\hspace{2mm}\bu\in \bH^{k+1}(\Omega),\label{AppPro3}\\
 \inf_{\bv_h\in \bV_h}\|\grad (\bu-\bv_h)\|&\leq Ch^{k}|\bu|_{k+1},\hspace{2mm}\bu\in \bH^{k+1}(\Omega),\label{AppPro4}
\end{align}
where $P^{L^2}_{\bV_h}(\bu)$ is the $L^2$ projection of $\bu$ into $\bV_h$.

\textcolor{black}{The following lemma for the discrete Gr\"onwall inequality was given in \cite{HR90}.
\begin{lemma}\label{dgl}
Let $\Delta t$, $\cD$, $a_n$, $b_n$, $c_n$, $d_n$ be non-negative numbers for $n=1,\cdots\hspace{-0.35mm},M$ such that
    $$a_M+\Delta t \sum_{n=1}^Mb_n\leq \Delta t\sum_{n=1}^{M-1}{d_na_n}+\Delta t\sum_{n=1}^Mc_n+\cD\hspace{3mm}\mbox{for}\hspace{2mm}M\in\mathbb{N},$$
then for all $\Delta t> 0,$
$$a_M+\Delta t\sum_{n=1}^Mb_n\leq \mbox{exp}\left(\Delta t\sum_{n=1}^{M-1}d_n\right)\lp\Delta t\sum_{n=1}^Mc_n+\cD\rp\hspace{2mm}\mbox{for}\hspace{2mm}M\in\mathbb{N}.$$
\end{lemma}}

\section{\large Fully discrete scheme and analysis}\label{fully-discrete-scheme}
Now we present and analyze an efficient, fully discrete, and decoupled time-stepping scheme for computing MHD flow ensembles. 
The scheme is defined below.

 \begin{algorithm}[H]\label{Algn1}
  \caption{Fully discrete and decoupled ensemble scheme} Given time-step $\Delta t>0$, end time $T>0$, initial conditions $\bv_j^0, \bw_j^0\in \bV_h$ and $\bif_{1,j}, \bif_{2,j}\in$ $ L^\infty\left( 0,T;\bH^{-1}(\Omega)\right)$ for $j=1,2,\cdots\hspace{-0.35mm},J$. Set $M=T/\Delta t$ and for $n=1,\cdots\hspace{-0.35mm},M-1$, compute:
 Find $\bv_{j,h}^{n+1}\in \bV_h$ satisfying, for all $\bchi_{j,h}\in \bV_h$:
 \begin{align}
\Bigg(\frac{\bv_{j,h}^{n+1}-\bv_{j,h}^n}{\Delta t}, &\bchi_{j,h}\Bigg)+b\left(<\bw_h>^n, \bv_{j,h}^{n+1},\bchi_{j,h}\right)+\frac{\Bar{\nu}+\Bar{\nu}_{m}}{2}\left(\nabla \bv_{j,h}^{n+1},\nabla \bchi_{j,h}\right)\nonumber\\&+\left(2\nu_T(\bw^{'}_{h},t^n)\nabla \bv_{j,h}^{n+1},\nabla \bchi_{j,h}\right)= \left(\bif_{1,j}(t^{n+1}),\bchi_{j,h}\right)-b\left(\bw_{j,h}^{'n}, \bv_{j,h}^n,\bchi_{j,h}\right)\nonumber\\&-\frac{\nu_j-\nu_{m,j}}{2}\left(\nabla \bw_{j,h}^n,\nabla \bchi_{j,h}\right)-\frac{\nu_j^{'}+\nu_{m,j}^{'}}{2}\left(\nabla \bv_{j,h}^{n},\nabla\bchi_{j,h}\right).\label{weaknew1}
\end{align}
Find $\bw_{j,h}^{n+1}\in \bV_h$ satisfying, for all $\bl_{j,h}\in \bV_h$:

\begin{align}
\Bigg(\frac{\bw_{j,h}^{n+1}-\bw_{j,h}^n}{\Delta t},\bl_{j,h}&\Bigg)+b\left(<\bv_h>^n, \bw_{j,h}^{n+1},\bl_{j,h}\right)+\frac{\Bar{\nu}+\Bar{\nu}_{m}}{2}\left(\nabla \bw_{j,h}^{n+1},\nabla \bl_{j,h}\right)\nonumber\\+&\left(2\nu_T(\bv^{'}_{h},t^n)\nabla \bw_{j,h}^{n+1},\nabla \bl_{j,h}\right)= \left(\bif_{2,j}(t^{n+1}),\bl_{j,h}\right)-b\left(\bv_{j,h}^{'n}, \bw_{j,h}^n,\bl_{j,h}\right)\nonumber\\-&\frac{\nu_j-\nu_{m,j}}{2}\left(\nabla \bv_{j,h}^n,\nabla \bl_{j,h}\right)-\frac{\nu_j^{'}+\nu_{m,j}^{'}}{2}\left(\nabla \bw_{j,h}^{n},\nabla \bl_{j,h}\right).\label{weaknew2}
\end{align}
\end{algorithm}

\subsection{Stability analysis}\label{stability-analysis}
We now prove stability and well-posedness for the Algorithm \ref{Algn1}. To simplify \textcolor{black}{the} notation, denote $\alpha_j:=\Bar{\nu}+\Bar{\nu}_{m}-|\nu_j-\nu_{m,j}|-|\nu_j^{'}+\nu_{m,j}^{'}|$, \textcolor{black}{for $j=1,2,\cdots\hspace{-0.35mm}, J$}.

\begin{theorem} Suppose $\textbf{f}_{1,j},\textbf{f}_{2,j}\in L^\infty\left(0,T;\bH^{-1}(\Omega)\right)$, and $\bv_{j,h}^0$, $\bw_{j,h}^0 \in \bH^1(\Omega)$, then the solutions to the Algorithm \ref{Algn1} are stable: For any $\Delta t>0$, if $\alpha_j>0$, and $\mu>\frac{1}{2}$
\begin{align*}
    \|\bv_{j,h}^M\|^2+\|\bw_{j,h}^M\|^2+\frac{\Bar{\nu}+\Bar{\nu}_{m}}{2}\Delta t\Big(\|\nabla \bv_{j,h}^{M}\|^2+\|\nabla \bw_{j,h}^{M}\|^2\Big)+\frac{\alpha_j\Delta t}{2}\sum_{n=0}^{M-1}\Big(\|\nabla \bv_{j,h}^n\|^2+\|\nabla \bw_{j,h}^n\|^2\Big)\nonumber\\\le\|\bv_{j,h}^0\|^2+\|\bw_{j,h}^0\|^2+\frac{\Bar{\nu}+\Bar{\nu}_{m}}{2}\Delta t\Big(\|\nabla \bv_{j,h}^{0}\|^2+\|\nabla \bw_{j,h}^{0}\|^2\Big)\nonumber\\+\frac{2\Delta t}{\alpha_j}\sum_{n=0}^{M-1}\Big(\|\textbf{f}_{1,j}(t^{n+1})\|_{-1}^2+\|\textbf{f}_{2,j}(t^{n+1})\|_{-1}^2\Big).
\end{align*}
\end{theorem}

\begin{proof}
Choose $\bchi_{j,h}=\bv_{j,h}^{n+1}$ and $\bl_{j,h}=\bw_{j,h}^{n+1}$ in \eqref{weaknew1}-\eqref{weaknew2} \textcolor{black}{to obtain}
\begin{align}
\Bigg(\frac{\bv_{j,h}^{n+1}-\bv_{j,h}^n}{\Delta t},\bv_{j,h}^{n+1}\Bigg)+&\left(\bw_{j,h}^{'n}\cdot\nabla \bv_{j,h}^n,\bv_{j,h}^{n+1}\right)+\frac{\Bar{\nu}+\Bar{\nu}_{m}}{2}\|\nabla \bv_{j,h}^{n+1}\|^2+\lp2\nu_T(\bw^{'}_{h},t^n)\nabla \bv_{j,h}^{n+1},\nabla \bv_{j,h}^{n+1}\rp\nonumber\\= \Big(\bif_{1,j}(t^{n+1}),&\bv_{j,h}^{n+1}\Big)-\frac{\nu_j-\nu_{m,j}}{2}\lp\nabla \bw_{j,h}^n,\nabla \bv_{j,h}^{n+1}\rp-\frac{\nu_j^{'}+\nu_{m,j}^{'}}{2}\left(\nabla \bv_{j,h}^{n},\nabla \bv_{j,h}^{n+1}\right),\label{weakn1}
\end{align}
and
\begin{align}
\Bigg(\frac{\bw_{j,h}^{n+1}-\bw_{j,h}^n}{\Delta t},\bw_{j,h}^{n+1}\Bigg)+&\left(\bv_{j,h}^{'n}\cdot\nabla \bw_{j,h}^n,\bw_{j,h}^{n+1}\right)+\frac{\Bar{\nu}+\Bar{\nu}_{m}}{2}\|\nabla \bw_{j,h}^{n+1}\|^2+\lp2\nu_T(\bv^{'}_{h},t^n)\nabla \bw_{j,h}^{n+1},\nabla \bw_{j,h}^{n+1}\rp\nonumber\\ = \Big(\bif_{2,j}(t^{n+1}),&\bw_{j,h}^{n+1}\Big)-\frac{\nu_j-\nu_{m,j}}{2}\lp\nabla \bv_{j,h}^n,\nabla \bw_{j,h}^{n+1}\rp-\frac{\nu_j^{'}+\nu_{m,j}^{'}}{2}\left(\nabla \bw_{j,h}^{n},\nabla \bw_{j,h}^{n+1}\right).\label{weakn2}
\end{align}
Using the polarization identity and that $(2\nu_T(\bw^{'}_{h},t^n)\nabla \bv_{j,h}^{n+1},\nabla \bv_{j,h}^{n+1})=2\mu\Delta t\|l^{n}_{w,h}\nabla \bv_{j,h}^{n+1}\|^2$, we get 
\begin{align}
    \frac{1}{2\Delta t }&\left(\|\bv_{j,h}^{n+1}-\bv_{j,h}^n\|^2+\|\bv_{j,h}^{n+1}\|^2-\|\bv_{j,h}^{n}\|^2\right)+\left(\bw_{j,h}^{'n}\cdot\nabla \bv_{j,h}^n,\bv_{j,h}^{n+1}\right)\nonumber\\&+\frac{\Bar{\nu}+\Bar{\nu}_{m}}{2}\|\nabla \bv_{j,h}^{n+1}\|^2+2\mu\Delta t\|l^{n}_{w,h}\nabla \bv_{j,h}^{n+1}\|^2=\Big(\bif_{1,j}(t^{n+1}),\bv_{j,h}^{n+1}\Big)\nonumber\\&-\frac{\nu_j-\nu_{m,j}}{2}\lp\nabla \bw_{j,h}^n,\nabla \bv_{j,h}^{n+1}\rp-\frac{\nu_j^{'}+\nu_{m,j}^{'}}{2}\left(\nabla \bv_{j,h}^{n},\nabla \bv_{j,h}^{n+1}\right), \label{pol11}
\end{align}
and 
\begin{align}
    \frac{1}{2\Delta t }&\left(\|\bw_{j,h}^{n+1}-\bw_{j,h}^n\|^2+\|\bw_{j,h}^{n+1}\|^2-\|\bw_{j,h}^{n}\|^2\right)+\left(\bv_{j,h}^{'n}\cdot\nabla \bw_{j,h}^n,\bw_{j,h}^{n+1}\right)\nonumber\\&+\frac{\Bar{\nu}+\Bar{\nu}_m}{2}\|\nabla \bw_{j,h}^{n+1}\|^2+2\mu\Delta t\|l^{n}_{v,h}\nabla \bw_{j,h}^{n+1}\|^2=\Big(\bif_{2,j}(t^{n+1}),\bw_{j,h}^{n+1}\Big)\nonumber\\&-\frac{\nu_j-\nu_{m,j}}{2}\lp\nabla \bv_{j,h}^n,\nabla \bw_{j,h}^{n+1}\rp-\frac{\nu_j^{'}+\nu_{m,j}^{'}}{2}\left(\nabla \bw_{j,h}^{n},\nabla \bw_{j,h}^{n+1}\right).\label{pol22}
\end{align}
Adding \eqref{pol11} and \eqref{pol22}, using inequality $\|\ba\cdot\nabla\bb\|\le\||\ba|\nabla \bb\|$ in
\begin{align*}
    \left(\bw_{j,h}^{'n}\cdot\nabla \bv_{j,h}^n,\bv_{j,h}^{n+1}\right)&=-\left(\bw_{j,h}^{'n}\cdot\nabla \bv_{j,h}^{n+1}, \bv_{j,h}^n\right)\\&=\left(\bw_{j,h}^{'n}\cdot\nabla \bv_{j,h}^{n+1}, \bv_{j,h}^{n+1}-\bv_{j,h}^n\right)\\&\le\|\bw_{j,h}^{'n}\cdot\nabla \bv_{j,h}^{n+1}\|\hspace{1mm}\|\bv_{j,h}^{n+1}-\bv_{j,h}^n\|\\&\le\||\bw_{j,h}^{'n}|\nabla \bv_{j,h}^{n+1}\|\hspace{1mm}\|\bv_{j,h}^{n+1}-\bv_{j,h}^n\|\nonumber\\&\le\|l^{n}_{w,h}\nabla \bv_{j,h}^{n+1}\|\hspace{1mm}\|\bv_{j,h}^{n+1}-\bv_{j,h}^n\|,
\end{align*}
and \textcolor{black}{after applying the} Cauchy-Schwarz inequality, reduces to

\begin{align}
    \frac{1}{2\Delta t }\Big(&\|\bv_{j,h}^{n+1}\|^2-\|\bv_{j,h}^{n}\|^2+\|\bw_{j,h}^{n+1}\|^2-\|\bw_{j,h}^{n}\|^2+\|\bv_{j,h}^{n+1}-\bv_{j,h}^n\|^2+\|\bw_{j,h}^{n+1}-\bw_{j,h}^n\|^2\Big)\nonumber\\&+\frac{\Bar{\nu}+\Bar{\nu}_m}{2}\left(\|\nabla \bv_{j,h}^{n+1}\|^2+\|\nabla \bw_{j,h}^{n+1}\|^2\right)+2\mu\Delta t\left(\|l^{n}_{w,h}\nabla \bv_{j,h}^{n+1}\|^2+\|l^{n}_{v,h}\nabla \bw_{j,h}^{n+1}\|^2\right)\nonumber\\&\le\|l^{n}_{w,h}\nabla \bv_{j,h}^{n+1}\|\hspace{1mm}\|\bv_{j,h}^{n+1}-\bv_{j,h}^n\|+\|l^{n}_{v,h}\nabla \bw_{j,h}^{n+1}\|\hspace{1mm}\|\bw_{j,h}^{n+1}-\bw_{j,h}^n\|\nonumber\\&+\|\bif_{1,j}(t^{n+1})\|_{-1}\|\nabla \bv_{j,h}^{n+1}\|+\|\bif_{2,j}(t^{n+1})\|_{-1}\|\nabla \bw_{j,h}^{n+1}\|\nonumber\\&+\frac{|\nu_j^{'}+\nu_{m,j}^{'}|}{2}\Big(\|\nabla \bv_{j,h}^{n+1}\|\|\nabla \bv_{j,h}^{n}\|+\|\nabla \bw_{j,h}^{n}\|\|\nabla \bw_{j,h}^{n+1}\|\Big)\nonumber\\&+\frac{|\nu_j-\nu_{m,j}|}{2}\Big(\|\nabla \bv_{j,h}^{n+1}\|\|\nabla \bw_{j,h}^{n}\|+\|\nabla \bw_{j,h}^{n+1}\|\|\nabla \bv_{j,h}^{n}\|\Big).
\end{align}

Using Young's inequality and \textcolor{black}{reducing}, we have
\begin{align}
    \frac{1}{2\Delta t }\Big(&\|\bv_{j,h}^{n+1}\|^2-\|\bv_{j,h}^{n}\|^2+\|\bw_{j,h}^{n+1}\|^2-\|\bw_{j,h}^{n}\|^2\Big)+\frac{1}{4\Delta t}\Big(\|\bv_{j,h}^{n+1}-\bv_{j,h}^n\|^2+\|\bw_{j,h}^{n+1}-\bw_{j,h}^n\|^2\Big)\nonumber\\&+\frac{\Bar{\nu}+\Bar{\nu}_m}{4}\left(\|\nabla \bv_{j,h}^{n+1}\|^2+\|\nabla \bw_{j,h}^{n+1}\|^2\right)+(2\mu-1)\Delta t\left(\|l^{n}_{w,h}\nabla \bv_{j,h}^{n+1}\|^2+\|l^{n}_{v,h}\nabla \bw_{j,h}^{n+1}\|^2\right)\nonumber\\&\le\frac{|\nu_j-\nu_{m,j}|+|\nu_j^{'}+\nu_{m,j}^{'}|}{4}\Big(\|\nabla \bv_{j,h}^{n}\|^2+\|\nabla \bw_{j,h}^{n}\|^2\Big)\nonumber\\&+\frac{1}{\alpha_j}\Big(\|\bif_{1,j}(t^{n+1})\|_{-1}^2+\|\bif_{2,j}(t^{n+1})\|_{-1}^2\Big).
\end{align}
Assuming $\mu>\frac12$, and dropping non-negative terms from the left-\textcolor{black}{hand-side}, this reduces to

\begin{align}
    \frac{1}{2\Delta t }\Big(&\|\bv_{j,h}^{n+1}\|^2-\|\bv_{j,h}^{n}\|^2+\|\bw_{j,h}^{n+1}\|^2-\|\bw_{j,h}^{n}\|^2\Big)\nonumber\\&+\frac{\Bar{\nu}+\Bar{\nu}_m}{4}\left(\|\nabla \bv_{j,h}^{n+1}\|^2-\|\nabla \bv_{j,h}^{n}\|^2+\|\nabla \bw_{j,h}^{n+1}\|^2-\|\nabla \bw_{j,h}^{n}\|^2\right)\nonumber\\&+\frac{\alpha_j}{4}\big(\|\nabla \bv_{j,h}^n\|^2+\|\nabla \bw_{j,h}^n\|^2\big)\le\frac{1}{\alpha_j}\left(\|\bif_{1,j}(t^{n+1})\|_{-1}^2+\|\bif_{2,j}(t^{n+1})\|_{-1}^2\right).
\end{align}
Multiplying both sides by $2\Delta t$, and summing over time-steps $n=0,\cdots\hspace{-0.35mm},M-1$, completes the proof.

\end{proof}
\textcolor{black}{\begin{remark}
The Algorithm \ref{Algn1} is finite dimensional and linear at each time-step, thus the above stability implies the well-posedness of the scheme. Again, due to the linearity, the stability provides uniqueness, and uniqueness implies existence.
\end{remark}}
\subsection{Convergence}
\textcolor{black}{We now prove the convergence of the proposed decoupled and unconditionally stable Algorithm \ref{Algn1}, which converges in space and time, provided that the true solution is sufficiently smooth.}

\begin{theorem}\label{convergence-theorem}
Assume $\left(\bv_j, \bw_j, q_j, r_j\right)$ satisfying \eqref{els1}-\eqref{els3} with regularity assumptions $\bv_j, \bw_j\in L^\infty(0,T;\bH^{k+1}(\Omega))$, $\bv_{j,t}, \bw_{j,t}\in L^\infty(0,T;\bH^{2}(\Omega))$, $\bv_{j,tt}, \bw_{j,tt}\in L^\infty(0,T;\bL^{2}(\Omega))$ for $j=1,2,\cdots\hspace{-0.35mm},J$, then the ensemble average solution $(<\bv_h>,<\bw_h>)$ to the Algorithm \ref{Algn1} converges to the true ensemble average solution: For any $\Delta t>0$, if $\alpha_j>0$, and $\mu>\frac12$, one has
\begin{align}
    \|\hspace{-1mm}<\hspace{-1mm}\bv\hspace{-1mm}>(T)-<\hspace{-1mm}\bv_h\hspace{-1mm}>^M\hspace{-1mm}\|^2+\|\hspace{-1mm}<\hspace{-1mm}\bw\hspace{-1mm}>(T)-<\hspace{-1mm}\bw_h\hspace{-1mm}>^M\hspace{-1mm}\|^2+\frac{\alpha_j\Delta t}{2}\sum_{n=1}^{M}\Big(\|\nabla\big(\hspace{-1mm}<\hspace{-1mm} \bv\hspace{-1mm}>(t^{n})-<\hspace{-1mm}\bv_h\hspace{-1mm}>^n\big)\|^2\nonumber\\+\|\nabla(<\hspace{-1mm} \bw\hspace{-1mm}>(t^{n})-<\hspace{-1mm}\bw_h\hspace{-1mm}>^n)\|^2\Big)\le C \exp\left(\frac{CT}{\alpha_j}\left(1+\frac{\Delta t^2}{J}\right)\right)\Big(\Delta t^2+h^{2k}+h^{2k}\Delta t^2\nonumber\\+h^{2-d}\Delta t^2+ h^{2k-1}\Delta t+h^{2k+2}\Big).
\end{align}
\end{theorem}

\begin{remark}
In 3D, the predicted temporal convergence rate could be reduced to $O(\Delta t(1+h^{-1/2}))$, which is less than the optimal rate  $O(\Delta t)$. This reduction happens due to the use of inverse inequality in the analysis of the stabilization term. It can be improved to $O(\Delta t)$ without using the stabilization term in the scheme, but that will incur a time-step restriction for the stability and convergence theorems.
\end{remark}

\begin{proof}
We start our proof by obtaining the error equations. Testing \eqref{els1} and \eqref{els2} with $\bchi_{j,h}, \bl_{j,h}\in \bV_h$ at the time level $t^{n+1}$, the continuous variational formulations can be written as
\begin{align}
\bigg(&\frac{\bv_j(t^{n+1})-\bv_j(t^n)}{\Delta t},\bchi_{j,h}\bigg)+\Big(\bw_j(t^{n+1})\cdot\nabla \bv_j(t^{n+1}),\bchi_{j,h}\Big)+\frac{\Bar{\nu}+\Bar{\nu}_{m}}{2}\Big(\nabla \bv_j(t^{n+1}), \nabla\bchi_{j,h}\Big) \nonumber\\ &=\Big(\bif_{1,j}(t^{n+1}),\bchi_{j,h}\Big)-\frac{\nu_j^{'}+\nu_{m,j}^{'}}{2}\Big(\nabla \bv_j(t^{n+1}), \nabla\bchi_{j,h}\Big)-\frac{\nu_j-\nu_{m,j}}{2}\Big(\nabla \bw_j(t^{n}),\nabla\bchi_{j,h}\Big)\nonumber\\&-\frac{\nu_j-\nu_{m,j}}{2}\Big(\nabla\big( \bw_j(t^{n+1})-\bw_j(t^{n})\big),\nabla\bchi_{j,h}\Big)-\bigg(\bv_{j,t}(t^{n+1})-\frac{\bv_j(t^{n+1})-\bv_j(t^n)}{\Delta t}, \bchi_{j,h}\bigg), \label{conweakn1}
\end{align}
and
\begin{align}
\bigg(&\frac{\bw_j(t^{n+1})-\bw_j(t^n)}{\Delta t},\bl_{j,h}\bigg)+\Big(\bv_j(t^{n+1})\cdot\nabla \bw_j(t^{n+1}),\bl_{j,h}\Big)+\frac{\Bar{\nu}+\Bar{\nu}_{m}}{2}\Big(\nabla \bw_j(t^{n+1}), \nabla \bl_{j,h}\Big) \nonumber\\ &=\Big(\bif_{2,j}(t^{n+1}),\bl_{j,h}\Big)-\frac{\nu_j^{'}+\nu_{m,j}^{'}}{2}\Big(\nabla \bw_j(t^{n+1}), \nabla \bl_{j,h}\Big)-\frac{\nu_j-\nu_{m,j}}{2}\Big(\nabla \bv_j(t^{n}),\nabla \bl_{j,h}\Big)\nonumber\\&-\frac{\nu_j-\nu_{m,j}}{2}\Big(\nabla\big( \bv_j(t^{n+1})-\bv_j(t^{n})\big),\nabla \bl_{j,h}\Big)-\bigg(\bw_{j,t}(t^{n+1})-\frac{\bw_j(t^{n+1})-\bw_j(t^n)}{\Delta t}, \bl_{j,h}\bigg). \label{conweakn2}
\end{align}

Denote $\be_{\bv,j}^n:=\bv_j(t^n)-\bv_{j,h}^n,\hspace{2mm}\be_{\bw,j}^n:=\bw_j(t^n)-\bw_{j,h}^n.$ Subtracting \eqref{weaknew1} and \eqref{weaknew2} from equation \eqref{conweakn1} and \eqref{conweakn2}, respectively, yields 
\begin{align}
\bigg(\frac{\be_{\bv,j}^{n+1}-\be_{\bv,j}^n}{\Delta t},\bchi_{j,h}\bigg)+\Big(\hspace{-1mm}<\hspace{-1mm}\be_{\bw}\hspace{-1mm}>^n\cdot\nabla\big(\bv_j(t^{n+1})-\bv_j(t^{n})\big),\bchi_{j,h}\Big)+\Big(\hspace{-1mm}<\hspace{-1mm}\bw_h\hspace{-1mm}>^n\cdot\nabla \be_{\bv,j}^{n+1},\bchi_{j,h}\Big)\nonumber\\+\Big(\bw_{j,h}^{'n}\cdot \nabla \be_{\bv,j}^n,\bchi_{j,h}\Big)+\Big(\be_{\bw,j}^n\cdot\nabla \bv_j(t^n),\bchi_{j,h}\Big)+\frac{\nu_j-\nu_{m,j}}{2}\Big(\nabla \be_{\bw,j}^n,\nabla\bchi_{j,h}\Big)\nonumber\\+\frac{\Bar{\nu}+\Bar{\nu}_{m}}{2}\Big(\nabla \be_{\bv,j}^{n+1},\nabla \bchi_{j,h}\Big)+\frac{\nu_j^{'}+\nu_{m,j}^{'}}{2}\Big(\nabla \be_{\bv,j}^n,\nabla\bchi_{j,h}\Big)-2\mu\Delta t\Big((l^{n}_{w,h})^2\nabla \bv_{j}(t^{n+1}),\nabla\bchi_{j,h}\Big)\nonumber\\+2\mu\Delta t\Big((l^{n}_{w,h})^2\nabla \be_{\bv,j}^{n+1},\nabla\bchi_{j,h}\Big)=-G_1(t,\bv_j,\bw_j,\bchi_{j,h}),
\end{align}
and
\begin{align}
\bigg(\frac{\be_{\bw,j}^{n+1}-\be_{\bw,j}^n}{\Delta t},\bl_{j,h}\bigg)+\Big(\hspace{-1mm}<\hspace{-1mm}\be_{\bv}\hspace{-1mm}>^n\cdot\nabla\big(\bw_j(t^{n+1})-\bw_j(t^{n})\big),\bl_{j,h}\Big)+\Big(\hspace{-1mm}<\hspace{-1mm}\bv_h\hspace{-1mm}>^n\cdot\nabla \be_{\bw,j}^{n+1},\bl_{j,h}\Big)\nonumber\\+\Big(\bv_{j,h}^{'n}\cdot \nabla \be_{\bw,j}^n,\bl_{j,h}\Big)+\Big(\be_{\bv,j}^n\cdot\nabla \bw_j(t^n),\bl_{j,h}\Big)+\frac{\nu_j-\nu_{m,j}}{2}\Big(\nabla \be_{\bv,j}^n,\nabla \bl_{j,h}\Big)\nonumber\\+\frac{\Bar{\nu}+\Bar{\nu}_{m}}{2}\Big(\nabla \be_{\bw,j}^{n+1},\nabla \bl_{j,h}\Big)+\frac{\nu_j^{'}+\nu_{m,j}^{'}}{2}\Big(\nabla \be_{\bw,j}^n,\nabla \bl_{j,h}\Big)-2\mu\Delta t\Big((l^{n}_{v,h})^2\nabla \bw_{j}(t^{n+1}),\nabla \bl_{j,h}\Big)\nonumber\\+2\mu\Delta t\Big((l^{n}_{v,h})^2\nabla \be_{\bw,j}^{n+1},\nabla \bl_{j,h}\Big)=-G_2(t,\bv_j,\bw_j,\bl_{j,h}),
\end{align}
where
\begin{align}
    G_1(t,\bv_j,&\bw_j,\bchi_{j,h}):=\bigg(\bv_{j,t}(t^{n+1})-\frac{\bv_j(t^{n+1})-\bv_j(t^n)}{\Delta t}, \bchi_{j,h}\bigg)+\frac{\nu_j^{'}+\nu_{m,j}^{'}}{2}\Big(\nabla \big(\bv_j(t^{n+1})-\bv_j(t^{n})\big),\nabla\bchi_{j,h}\Big)\nonumber\\&+\Big(\big(\bw_j(t^{n+1})-\bw_j(t^{n})\big)\cdot\nabla \bv_j(t^{n+1}),\bchi_{j,h}\Big)+\frac{\nu_j-\nu_{m,j}}{2}\Big(\nabla\big( \bw_j(t^{n+1})-\bw_j(t^{n})\big),\nabla\bchi_{j,h}\Big)\nonumber\\&+\Big(\big(\bw_j(t^n)-<\hspace{-1mm}\bw(t^n)\hspace{-1mm}>\hspace{-1mm}\big)\cdot\nabla\big(\bv_j(t^{n+1})-\bv_j(t^{n})\big),\bchi_{j,h}\Big),
\end{align}
and
\begin{align}
    G_2(t,\bv_j,&\bw_j,\bl_{j,h}):=\bigg(\bw_{j,t}(t^{n+1})-\frac{\bw_j(t^{n+1})-\bw_j(t^n)}{\Delta t}, \bl_{j,h}\bigg)+\frac{\nu_j^{'}+\nu_{m,j}^{'}}{2}\Big(\nabla \big(\bw_j(t^{n+1})-\bw_j(t^{n})\big),\nabla \bl_{j,h}\Big)\nonumber\\&+\Big(\big(\bv_j(t^{n+1})-\bv_j(t^{n})\big)\cdot\nabla \bw_j(t^{n+1}),\bl_{j,h}\Big)+\frac{\nu_j-\nu_{m,j}}{2}\Big(\nabla\big( \bv_j(t^{n+1})-\bv_j(t^{n})\big),\nabla \bl_{j,h}\Big)\nonumber\\&+\Big(\big(\bv_j(t^n)-<\hspace{-1mm}\bv(t^n)\hspace{-1mm}>\hspace{-1mm}\big)\cdot\nabla\big(\bw_j(t^{n+1})-\bw_j(t^{n})\big),\bl_{j,h}\Big).
\end{align}
Now we decompose the errors as
\begin{align*}
    \be_{\bv,j}^n:& = \bv_j(t^n)-\bv_{j,h}^n=(\bv_j(t^n)-\tilde{\bv}_j^n)-(\bv_{j,h}^n-\tilde{\bv}_j^n):=\bfeta_{\bv,j}^n-\bphi_{j,h}^n,\\
    \be_{\bw,j}^n: &= \bw_j(t^n)-\bw_{j,h}^n=(\bw_j(t^n)-\tilde{\bw}_j^n)-(\bw_{j,h}^n-\tilde{\bw}_j^n):=\eta_{\bw,j}^n-\bpsi_{j,h}^n,
\end{align*}
where $\tilde{\bv}_j^n: =P_{\bV_h}^{L^2}(\bv_j(t^n))\in \bV_h$ and $\tilde{\bw}_j^n: =P_{\bV_h}^{L^2}(\bw_j(t^n))\in \bV_h$ are the $L^2$ projections of $\bv_j(t^n)$ and $\bw_j(t^n)$ into $\bV_h$, respectively. Note that $(\bfeta_{\bv,j}^n,\bv_{j,h})=(\bfeta_{\bw,j}^n,\bv_{j,h})=0\hspace{2mm} \forall \bv_{j,h}\in \bV_h.$  Rewriting, we have for $\bchi_{j,h}, \bl_{j,h}\in \bV_h$

\begin{align}
\bigg(\frac{\bphi_{j,h}^{n+1}-\bphi_{j,h}^n}{\Delta t},\bchi_{j,h}\bigg)+\Big(\hspace{-1mm}<\hspace{-1mm}\bpsi_{h}\hspace{-1mm}>^n\cdot\nabla\big(\bv_j(t^{n+1})-\bv_j(t^{n})\big),\bchi_{j,h}\Big)+\Big(\hspace{-1mm}<\hspace{-1mm}\bw_h\hspace{-1mm}>^n\cdot\nabla \bphi_{j,h}^{n+1},\bchi_{j,h}\Big)\nonumber\\+\Big(\bw_{j,h}^{'n}\cdot \nabla\bphi_{j,h} ^n,\bchi_{j,h}\Big)+\Big(\bpsi_{j,h}^n\cdot\nabla \bv_j(t^n),\bchi_{j,h}\Big)+\frac{\nu_j-\nu_{m,j}}{2}\Big(\nabla\bpsi_{j,h}^n,\nabla\bchi_{j,h}\Big)\nonumber\\+\frac{\Bar{\nu}+\Bar{\nu}_{m}}{2}\Big(\nabla \bphi_{j,h}^{n+1},\nabla \bchi_{j,h}\Big)+\frac{\nu_j^{'}+\nu_{m,j}^{'}}{2}\Big(\nabla \bphi_{j,h}^n,\nabla\bchi_{j,h}\Big)+2\mu\Delta t\Big((l^{n}_{w,h})^2\nabla \bphi_{j,h}^{n+1},\nabla\bchi_{j,h}\Big)\nonumber\\=\Big(\hspace{-1mm}<\bfeta_{\bw}\hspace{-1mm}>^n\cdot\nabla\big(\bv_j(t^{n+1})-\bv_j(t^{n})\big),\bchi_{j,h}\Big)+\Big(\hspace{-1mm}<\hspace{-1mm}\bw_h\hspace{-1mm}>^n\cdot\nabla \bfeta_{\bv,j}^{n+1},\bchi_{j,h}\Big)+\Big(\bw_{j,h}^{'n}\cdot \nabla\bfeta_{v,j} ^n,\bchi_{j,h}\Big)\nonumber\\+\Big(\bfeta^n_{\bw,j}\cdot\nabla \bv_j(t^n),\bchi_{j,h}\Big)+\frac{\nu_j-\nu_{m,j}}{2}\Big(\nabla \bfeta_{\bw,j}^n,\nabla\bchi_{j,h}\Big)+\frac{\Bar{\nu}+\Bar{\nu}_{m}}{2}\Big(\nabla \bfeta_{\bv,j}^{n+1},\nabla \bchi_{j,h}\Big)\nonumber\\+\frac{\nu_j^{'}+\nu_{m,j}^{'}}{2}\Big(\nabla \bfeta_{\bv,j}^n,\nabla\bchi_{j,h}\Big)+2\mu\Delta t\Big((l^{n}_{w,h})^2\nabla \bv_{j}(t^{n+1}),\nabla\bchi_{j,h}\Big)\nonumber\\+2\mu\Delta t\Big((l^{n}_{w,h})^2\nabla \bfeta_{\bv,j}^{n+1},\nabla\bchi_{j,h}\Big)-G_1(t,\bv_j,\bw_j,\bchi_{j,h}),\label{phi2n}
\end{align}

and
\begin{align}
\bigg(\frac{\bpsi_{j,h}^{n+1}-\bpsi_{j,h}^n}{\Delta t},\bl_{j,h}\bigg)+\Big(\hspace{-1mm}<\hspace{-1mm}\bphi_{h}\hspace{-1mm}>^n\cdot\nabla\big(\bw_j(t^{n+1})-\bw_j(t^{n})\big),\bl_{j,h}\Big)+\Big(\hspace{-1mm}<\hspace{-1mm}\bv_h\hspace{-1mm}>^n\cdot\nabla \bpsi_{j,h}^{n+1},\bl_{j,h}\Big)\nonumber\\+\Big(\bv_{j,h}^{'n}\cdot \nabla\bpsi_{j,h} ^n,\bl_{j,h}\Big)+\Big(\bphi_{j,h}^n\cdot\nabla \bw_j(t^n),\bl_{j,h}\Big)+\frac{\nu_j-\nu_{m,j}}{2}\Big(\nabla\bphi_{j,h}^n,\nabla \bl_{j,h}\Big)\nonumber\\+\frac{\Bar{\nu}+\Bar{\nu}_{m}}{2}\Big(\nabla \bpsi_{j,h}^{n+1},\nabla \bl_{j,h}\Big)+\frac{\nu_j^{'}+\nu_{m,j}^{'}}{2}\Big(\nabla \bpsi_{j,h}^n,\nabla \bl_{j,h}\Big)+2\mu\Delta t\Big((l^{n}_{v,h})^2\nabla \bpsi_{j,h}^{n+1},\nabla \bl_{j,h}\Big)\nonumber\\=\Big(\hspace{-1mm}<\hspace{-1mm}\bfeta_{\bv}\hspace{-1mm}>^n\cdot\nabla\big(\bw_j(t^{n+1})-\bw_j(t^{n})\big),\bl_{j,h}\Big)+\Big(\hspace{-1mm}<\hspace{-1mm}\bv_h\hspace{-1mm}>^n\cdot\nabla \bfeta_{\bw,j}^{n+1},\bl_{j,h}\Big)+\Big(\bv_{j,h}^{'n}\cdot \nabla\bfeta_{\bw,j} ^n,\bl_{j,h}\Big)\nonumber\\+\Big(\bfeta^n_{\bv,j}\cdot\nabla \bw_j(t^n),\bl_{j,h}\Big)+\frac{\nu_j-\nu_{m,j}}{2}\Big(\nabla \bfeta_{\bv,j}^n,\nabla \bl_{j,h}\Big)+\frac{\Bar{\nu}+\Bar{\nu}_{m}}{2}\Big(\nabla \bfeta_{\bw,j}^{n+1},\nabla \bl_{j,h}\Big)\nonumber\\+\frac{\nu_j^{'}+\nu_{m,j}^{'}}{2}\Big(\nabla \bfeta_{\bw,j}^n,\nabla \bl_{j,h}\Big)+2\mu\Delta t\Big((l^{n}_{v,h})^2\nabla \bw_{j}(t^{n+1}),\nabla \bl_{j,h}\Big)\nonumber\\+2\mu\Delta t\Big((l^{n}_{v,h})^2\nabla \bfeta_{\bw,j}^{n+1},\nabla \bl_{j,h}\Big)-G_2(t,\bv_j,\bw_j,\bl_{j,h}).\label{psi2n}
\end{align}
Choose $\bchi_{j,h}=\bphi_{j,h}^{n+1}, \bl_{j,h}=\bpsi_{j,h}^{n+1}$, and use the polarization identity in \eqref{phi2n} and \eqref{psi2n}, to obtain
\begin{align}
    &\frac{1}{2\Delta t}\lp\|\bphi_{j,h}^{n+1}\|^2-\|\bphi_{j,h}^{n}\|^2+\|\bphi_{j,h}^{n+1}-\bphi_{j,h}^{n}\|^2\rp+\frac{\Bar{\nu}+\Bar{\nu}_{m}}{2}\|\nabla \bphi_{j,h}^{n+1}\|^2+2\mu\Delta t\|l^{n}_{w,h}\nabla \bphi_{j,h}^{n+1}\|^2\nonumber\\&\le\frac{|\nu_j-\nu_{m,j}|}{2}\left|\lp\nabla\bpsi_{j,h}^n,\nabla\bphi_{j,h}^{n+1}\rp\right|+\frac{|\nu_j^{'}+\nu_{m,j}^{'}|}{2}\left|\lp\nabla \bphi_{j,h}^n,\nabla\bphi_{j,h}^{n+1}\rp\right|+\frac{|\nu_j-\nu_{m,j}|}{2}\left|\lp\nabla \bfeta_{\bw,j}^n,\nabla\bphi_{j,h}^{n+1}\rp\right|\nonumber\\&+\frac{\Bar{\nu}+\Bar{\nu}_{m}}{2}\left|\lp\nabla \bfeta_{\bv,j}^{n+1},\nabla \bphi_{j,h}^{n+1}\rp\right|+\frac{|\nu_j^{'}+\nu_{m,j}^{'}|}{2}\left|\lp\nabla \bfeta_{\bv,j}^n,\nabla\bphi_{j,h}^{n+1}\rp\right|+\left|\lp\bw_{j,h}^{'n}\cdot \nabla\bphi_{j,h} ^n,\bphi_{j,h}^{n+1}\rp\right|\nonumber\\&+2\mu\Delta t\left|\lp(l^{n}_{w,h})^2\nabla \bv_{j}(t^{n+1}),\nabla\bphi_{j,h}^{n+1}\rp\right|+2\mu\Delta t\left|\lp(l^{n}_{w,h})^2\nabla \bfeta_{\bv,j}^{n+1},\nabla\bphi_{j,h}^{n+1}\rp\right|\nonumber\\&+
\left|\lp<\hspace{-1mm}\bpsi_{h}\hspace{-1mm}>^n\cdot\nabla\big(\bv_j(t^{n+1})-\bv_j(t^{n})\big),\bphi_{j,h}^{n+1}\rp\right|+\left|\lp\bpsi_{j,h}^n\cdot\nabla \bv_j(t^n),\bphi_{j,h}^{n+1}\rp\right|\nonumber\\&+\left|\lp<\hspace{-1mm}\bfeta_{\bw}\hspace{-1mm}>^n\cdot\nabla\big(\bv_j(t^{n+1})-\bv_j(t^{n})\big),\bphi_{j,h}^{n+1}\rp\right|+\left|\lp<\hspace{-1mm}\bw_h\hspace{-1mm}>^n\cdot\nabla \bfeta_{\bv,j}^{n+1},\bphi_{j,h}^{n+1}\rp\right|\nonumber\\&+\left|\lp\bw_{j,h}^{'n}\cdot \nabla\bfeta_{\bv,j} ^n,\bphi_{j,h}^{n+1}\rp\right|+\left|\lp\bfeta^n_{\bw,j}\cdot\nabla \bv_j(t^n),\bphi_{j,h}^{n+1}\rp\right|+\left|G_1(t,\bv_j,\bw_j,\bphi_{j,h}^{n+1})\right|,\label{phibd}
\end{align}
and
\begin{align}
    &\frac{1}{2\Delta t}\lp\|\bpsi_{j,h}^{n+1}\|^2-\|\bpsi_{j,h}^{n}\|^2+\|\bpsi_{j,h}^{n+1}-\bpsi_{j,h}^{n}\|^2\rp+\frac{\Bar{\nu}+\Bar{\nu}_{m}}{2}\|\nabla \bpsi_{j,h}^{n+1}\|^2+2\mu\Delta t\|l^{n}_{v,h}\nabla \bpsi_{j,h}^{n+1}\|^2\nonumber\\&\le\frac{|\nu_j-\nu_{m,j}|}{2}\left|\lp\nabla\bphi_{j,h}^n,\nabla\bpsi_{j,h}^{n+1}\rp\right|+\frac{|\nu_j^{'}+\nu_{m,j}^{'}|}{2}\left|\left(\nabla \bpsi_{j,h}^n,\nabla\bpsi_{j,h}^{n+1}\right)\right|+\frac{|\nu_j-\nu_{m,j}|}{2}\left|\lp\nabla \bfeta_{\bv,j}^n,\nabla\bpsi_{j,h}^{n+1}\rp\right|\nonumber\\&+\frac{\Bar{\nu}+\Bar{\nu}_{m}}{2}\left|\lp\nabla \bfeta_{\bw,j}^{n+1},\nabla \bpsi_{j,h}^{n+1}\rp\right|+\frac{|\nu_j^{'}+\nu_{m,j}^{'}|}{2}\left|\left(\nabla \bfeta_{\bw,j}^n,\nabla\bpsi_{j,h}^{n+1}\right)\right|+\left|\lp \bv_{j,h}^{'n}\cdot \nabla\bpsi_{j,h} ^n,\bpsi_{j,h}^{n+1}\rp\right|\nonumber\\&+2\mu\Delta t\left|\lp(l^{n}_{v,h})^2\nabla \bw_{j}(t^{n+1}),\nabla\bpsi_{j,h}^{n+1}\rp\right|+2\mu\Delta t\left|\lp(l^{n}_{v,h})^2\nabla \bfeta_{\bw,j}^{n+1},\nabla\bpsi_{j,h}^{n+1}\rp\right|\nonumber\\&+
\left|\lp<\hspace{-1mm}\bphi_{h}\hspace{-1mm}>^n\cdot\nabla\big(\bw_j(t^{n+1})-\bw_j(t^{n})\big),\bpsi_{j,h}^{n+1}\rp\right|+\left|\lp\bphi_{j,h}^n\cdot\nabla \bw_j(t^n),\bpsi_{j,h}^{n+1}\rp\right|\nonumber\\&+\left|\lp<\hspace{-1mm}\bfeta_{\bv}\hspace{-1mm}>^n\cdot\nabla\big(\bw_j(t^{n+1})-\bw_j(t^{n})\big),\bpsi_{j,h}^{n+1}\rp\right|+\left|\lp<\hspace{-1mm}\bv_h\hspace{-1mm}>^n\cdot\nabla \bfeta_{\bw,j}^{n+1},\bpsi_{j,h}^{n+1}\rp\right|\nonumber\\&+\left|\lp\bv_{j,h}^{'n}\cdot \nabla\bfeta_{\bw,j} ^n,\bpsi_{j,h}^{n+1}\rp\right|+\left|\lp\bfeta^n_{\bv,j}\cdot\nabla \bw_j(t^n),\bpsi_{j,h}^{n+1}\rp\right|+\left|G_2(t,\bv_j,\bw_j,\bpsi_{j,h}^{n+1})\right|.\label{psibd}
\end{align}
Now, turn our attention to finding bounds on
the right side terms of \eqref{phibd} (the estimates on terms in \eqref{psibd} are similar). Applying Cauchy-Schwarz 
Young’s inequalities on the first five terms results in
\begin{align*}
    \frac{|\nu_j-\nu_{m,j}|}{2}\left|\lp\nabla\bpsi_{j,h}^n,\nabla\bphi_{j,h}^{n+1}\rp\right|&\le\frac{|\nu_j-\nu_{m,j}|}{4}\left(\|\nabla \bphi_{j,h}^{n+1}\|^2+\|\nabla\bpsi_{j,h}^n\|^2\right),\\\frac{|\nu_j^{'}+\nu_{m,j}^{'}|}{2}\left|\left(\nabla \bphi_{j,h}^n,\nabla\bphi_{j,h}^{n+1}\right)\right|&\le\frac{|\nu_j^{'}+\nu_{m,j}^{'}|}{4}\left(\|\nabla \bphi_{j,h}^{n+1}\|^2+\|\nabla\bphi_{j,h}^n\|^2\right),\\\frac{|\nu_j-\nu_{m,j}|}{2}\left|\lp\nabla \bfeta_{\bw,j}^n,\nabla\bphi_{j,h}^{n+1}\rp\right|&\le\frac{\alpha_j}{44}\|\nabla \bphi_{j,h}^{n+1}\|^2+\frac{11(\nu_j-\nu_{m,j})^2}{4\alpha_j}\|\nabla \bfeta_{\bw,j}^n\|^2,\nonumber\\\frac{\Bar{\nu}+\Bar{\nu}_{m}}{2}\left|\lp\nabla \bfeta_{\bv,j}^{n+1},\nabla \bphi_{j,h}^{n+1}\rp\right|&\le \frac{\alpha_j}{44}\|\nabla \bphi_{j,h}^{n+1}\|^2+\frac{11(\Bar{\nu}+\Bar{\nu}_{m})^2}{4\alpha_j}\|\nabla \bfeta_{\bv,j}^{n+1}\|^2,\\\frac{|\nu_j^{'}+\nu_{m,j}^{'}|}{2}\left|\left(\nabla \bfeta_{\bv,j}^n,\nabla\bphi_{j,h}^{n+1}\right)\right|&\le\frac{\alpha_j}{44}\|\nabla \bphi_{j,h}^{n+1}\|^2+\frac{11(\nu_j^{'}+\nu_{m,j}^{'})^2}{4\alpha_j}\|\nabla \bfeta_{\bv,j}^n\|^2.
\end{align*}
For the first nonlinear term, rearranging and applying Cauchy-Schwarz and Young’s inequalities
yields
\begin{align*}
    \left|\lp \bw_{j,h}^{'n}\cdot \nabla\bphi_{j,h} ^n,\bphi_{j,h}^{n+1}\rp\right|=\left|-\lp \bw_{j,h}^{'n}\cdot \nabla\bphi_{j,h} ^{n+1},\bphi_{j,h}^{n}\rp\right|&=\left|\lp \bw_{j,h}^{'n}\cdot \nabla\bphi_{j,h} ^{n+1},\bphi_{j,h} ^{n+1}-\bphi_{j,h}^{n}\rp\right|\\&\le\|\bw_{j,h}^{'n}\cdot \nabla\bphi_{j,h} ^{n+1}\|\|\bphi_{j,h} ^{n+1}-\bphi_{j,h}^{n}\|\\&\le\|l^{n}_{w,h} \nabla\bphi_{j,h} ^{n+1}\|\|\bphi_{j,h} ^{n+1}-\bphi_{j,h}^{n}\|\\&\le\frac{1}{4\Delta t}\|\bphi_{j,h}^{n+1}-\bphi_{j,h}^{n}\|^2+\Delta t\|l^{n}_{w,h} \nabla\bphi_{j,h} ^{n+1}\|^2.
\end{align*}
For the second nonlinear term, we apply H\"older's inequality and the regularity assumptions of the true
solution to get
\begin{align*}
    2\mu\Delta t\left|\lp(l^{n}_{w,h})^2\nabla \bv_{j}(t^{n+1}),\nabla\bphi_{j,h}^{n+1}\rp\right|&\le C\mu\Delta t\|\nabla \bv_{j}(t^{n+1})\|_{L^\infty} \|l^{n}_{w,h}\|_{L^4}^2\|\nabla\bphi_{j,h}^{n+1}\|\\&\le\frac{\alpha_j}{44}\|\nabla \bphi_{j,h}^{n+1}\|^2+C\frac{\mu^2\Delta t^2}{\alpha_j}\|l^{n}_{w,h}\|_{L^4}^4.
\end{align*}
For the third nonlinear term, we rearrange, and apply Cauchy-Schwarz and Young’s inequalities assuming $\mu>1/2$
 to obtain
\begin{align*}
    2\mu\Delta t\left|\lp (l^{n}_{w,h})^2\nabla \bfeta_{\bv,j}^{n+1},\nabla\bphi_{j,h}^{n+1}\rp\right|&= 2\mu\Delta t\left(l^{n}_{w,h}\nabla \bfeta_{\bv,j}^{n+1},l^{n}_{w,h}\nabla\bphi_{j,h}^{n+1}\right)\\&\le 2\mu\Delta t\|l^{n}_{w,h}\nabla \bfeta_{\bv,j}^{n+1}\|\|l^{n}_{w,h}\nabla\bphi_{j,h}^{n+1}\|\\&\le\frac{2\mu-1}{4}\Delta t\|l^{n}_{w,h}\nabla\bphi_{j,h}^{n+1}\|^2+\frac{4\mu^2\Delta t}{2\mu-1}\|l^{n}_{w,h}\nabla \bfeta_{\bv,j}^{n+1}\|^2.
\end{align*}
For the fourth and fifth nonlinear terms, we use H\"older's inequality, Sobolev embedding theorems, Poincar\'e and Young’s inequalities to reveal
\begin{align*}
    \left|\lp<\hspace{-1mm}\bpsi_{h}\hspace{-1mm}>^n\cdot\nabla\big(\bv_j(t^{n+1})-\bv_j(t^{n})\big),\bphi_{j,h}^{n+1}\rp\right|&\le C\|\hspace{-1mm}<\hspace{-1mm}\bpsi_{h}\hspace{-1mm}>^n\hspace{-1mm}\|\|\nabla\big(\bv_j(t^{n+1})-\bv_j(t^{n})\big)\|_{L^6}\|\bphi_{j,h}^{n+1}\|_{L^3}\\&\le C\|\hspace{-1mm}<\hspace{-1mm}\bpsi_{h}\hspace{-1mm}>^n\hspace{-1mm}\|\|\bv_j(t^{n+1})-\bv_j(t^{n})\|_{H^2}\|\bphi_{j,h}^{n+1}\|^{\frac{1}{2}}\|\nabla \bphi_{j,h}^{n+1}\|^{\frac{1}{2}}\\&\le C\|\hspace{-1mm}<\hspace{-1mm}\bpsi_{h}\hspace{-1mm}>^n\hspace{-1mm}\|\|\bv_j(t^{n+1})-\bv_j(t^{n})\|_{H^2}\|\nabla \bphi_{j,h}^{n+1}\|\\&\le \frac{\alpha_j}{44}\|\nabla \bphi_{j,h}^{n+1}\|^2+\frac{C}{\alpha_j}\Delta t^2\|\hspace{-1mm}<\hspace{-1mm}\bpsi_{h}\hspace{-1mm}>^n\hspace{-1mm}\|^2\|\bv_{j,t}(t^{*})\|_{H^2}^2,\\
    \left|\lp\bpsi_{j,h}^n\cdot\nabla \bv_j(t^n),\bphi_{j,h}^{n+1}\rp\right|&\le \frac{\alpha_j}{44}\|\nabla \bphi_{j,h}^{n+1}\|^2+\frac{C}{\alpha_j}\|\bpsi_{j,h}^n\|^2\|\bv_j(t^n)\|^2_{H^2}.
\end{align*}
For the sixth, seventh, eighth, and ninth nonlinear terms, apply Young’s inequalities with \eqref{nonlinearbound} to obtain 
\begin{align*}
    \lp<\hspace{-1mm}\bfeta_{\bw}\hspace{-1mm}>^n\cdot\nabla\big(\bv_j(t^{n+1})-\bv_j(t^{n})\big),\bphi_{j,h}^{n+1}\rp&\le C\|\nabla \hspace{-1mm}<\hspace{-1mm}\bfeta_{\bw}\hspace{-1mm}>^n\hspace{-1mm}\|\|\nabla\big(\bv_j(t^{n+1})-\bv_j(t^{n})\big)\|\|\nabla \bphi_{j,h}^{n+1}\|\\&\le \frac{\alpha_j}{44}\|\nabla \bphi_{j,h}^{n+1}\|^2+\frac{C}{\alpha_j}\Delta t^2\|\nabla\hspace{-1mm} <\hspace{-1mm}\bfeta_{\bw}\hspace{-1mm}>^n\hspace{-1mm}\|^2\|\nabla \bv_{j,t}(t^{**})\|^2, \\\left|\lp<\hspace{-1mm}\bw_h\hspace{-1mm}>^n\cdot\nabla \bfeta_{\bv,j}^{n+1},\bphi_{j,h}^{n+1}\rp\right|&\le C\|\nabla\hspace{-1mm}<\hspace{-1mm}\bw_h\hspace{-1mm}>^n\hspace{-1mm}\|\|\nabla \bfeta_{\bv,j}^{n+1}\|\|\nabla\bphi_{j,h}^{n+1}\|\\&\le \frac{\alpha_j}{44}\|\nabla \bphi_{j,h}^{n+1}\|^2+\frac{C}{\alpha_j}\|\nabla\hspace{-1mm}<\hspace{-1mm}\bw_h\hspace{-1mm}>^n\hspace{-1mm}\|^2\|\nabla \bfeta_{\bv,j}^{n+1}\|^2,\\\left|\lp \bw_{j,h}^{'n}\cdot \nabla\bfeta_{\bv,j} ^n,\bphi_{j,h}^{n+1}\rp\right|&\le C\|\nabla \bw_{j,h}^{'n}\|\|\nabla\bfeta_{\bv,j} ^n\|\|\nabla\bphi_{j,h}^{n+1}\|\\&\le \frac{\alpha_j}{44}\|\nabla \bphi_{j,h}^{n+1}\|^2+\frac{C}{\alpha_j}\|\nabla \bw_{j,h}^{'n}\|^2\|\nabla\bfeta_{\bv,j} ^n\|^2,\\\left|\lp\bfeta^n_{\bw,j}\cdot\nabla \bv_j(t^n),\bphi_{j,h}^{n+1}\rp\right|&\le C\|\nabla \bfeta^n_{\bw,j}\|\|\nabla \bv_j(t^n)\|\|\nabla\bphi_{j,h}^{n+1}\|\\&\le \frac{\alpha_j}{44}\|\nabla \bphi_{j,h}^{n+1}\|^2+\frac{C}{\alpha_j}\|\nabla \bfeta^n_{\bw,j}\|^2\|\nabla \bv_j(t^n)\|^2.
\end{align*} 
Using Taylor’s series, Cauchy-Schwarz and Young's inequalities, the last term is evaluated as
\begin{align*}
    \left|G_1(t,\bv_j,\bw_j, \bphi_{j,h}^{n+1})\right|\le\frac{\alpha_j}{44}\|\nabla\bphi_{j,h}^{n+1}\|^2+C\Delta t^2\Big(\|\bv_{j,tt}(t_1^*)\|^2+\|\nabla \bv_{j,t}(t_2^*)\|^2+\|\nabla \bw_{j,t}(t_3^{*})\|^2\|\nabla \bv_j(t^{n+1})\|^2\\+\|\nabla  \bw_{j,t}(t_4^{*})\|^2+\|\nabla \big(\bw_j(t^n)-<\hspace{-1mm}\bw(t^n)\hspace{-1mm}>\hspace{-1mm}\big)\|^2\|\nabla \bv_{j,t}(t_5^{*})\|^2\Big),
\end{align*}
with $t_1^*,t_2^*,t_3^*,t_4^*,t_5^*,\in [t^n,t^{n+1}]$. Using these estimates in \eqref{phibd} and reducing produces
\begin{align}
    \frac{1}{2\Delta t}\lp\|\bphi_{j,h}^{n+1}\|^2-\|\bphi_{j,h}^{n}\|^2\rp+\frac{1}{4\Delta t}\|\bphi_{j,h}^{n+1}-\bphi_{j,h}^{n}\|^2+\frac{\Bar{\nu}+\Bar{\nu}_{m}}{4}\|\nabla \bphi_{j,h}^{n+1}\|^2\nonumber\\+\frac{2\mu-1}{4}\Delta t\|l^{n}_{w,h}\nabla \bphi_{j,h}^{n+1}\|^2\le\frac{|\nu_j-\nu_{m,j}|}{4}\|\nabla\bpsi_{j,h}^n\|^2+\frac{|\nu_j^{'}+\nu_{m,j}^{'}|}{4}\|\nabla\bphi_{j,h}^n\|^2\nonumber\\+\frac{11(\nu_j-\nu_{m,j})^2}{4\alpha_j}\|\nabla \bfeta_{\bw,j}^n\|^2+\frac{11(\Bar{\nu}+\Bar{\nu}_{m})^2}{4\alpha_j}\|\nabla \bfeta_{\bv,j}^{n+1}\|^2+\frac{11(\nu_j^{'}+\nu_{m,j}^{'})^2}{4\alpha_j}\|\nabla \bfeta_{\bv,j}^n\|^2\nonumber\\+C\frac{\mu^2\Delta t^2}{\alpha_j}\|l^{n}_{w,h}\|_{L^4}^4+\frac{4\mu^2\Delta t}{2\mu-1}\|l^{n}_{w,h}\nabla \bfeta_{\bv,j}^{n+1}\|^2+\frac{C}{\alpha_j}\Delta t^2\|\hspace{-1mm}<\hspace{-1mm}\bpsi_{h}\hspace{-1mm}>^n\hspace{-1mm}\|^2\|\bv_{j,t}(t^{*})\|_{\bH^2}^2\nonumber\\+\frac{C}{\alpha_j}\|\bpsi_{j,h}^n\|^2\|\bv_j(t^n)\|^2_{\bH^2}+\frac{C}{\alpha_j}\Delta t^2\|\nabla\hspace{-1mm} <\hspace{-1mm}\bfeta_{\bw}\hspace{-1mm}>^n\hspace{-1mm}\|^2\|\nabla \bv_{j,t}(t^{**})\|^2+\frac{C}{\alpha_j}\|\nabla\hspace{-1mm}<\hspace{-1mm}\bw_h\hspace{-1mm}>^n\hspace{-1mm}\|^2\|\nabla \bfeta_{\bv,j}^{n+1}\|^2\nonumber\\+\frac{C}{\alpha_j}\|\nabla \bw_{j,h}^{'n}\|^2\|\nabla\bfeta_{\bv,j} ^n\|^2+\frac{C}{\alpha_j}\|\nabla \bfeta^n_{\bw,j}\|^2\|\nabla \bv_j(t^n)\|^2+C\Delta t^2\Big(\|\bv_{j,tt}(t_1^*)\|^2+\|\nabla \bv_{j,t}(t_2^*)\|^2\nonumber\\+\|\nabla \bw_{j,t}(t_3^{*})\|^2\|\nabla \bv_j(t^{n+1})\|^2+\|\nabla  \bw_{j,t}(t_4^{*})\|^2+\|\nabla (\bw_j(t^n)-<\hspace{-1mm}\bw(t^n)\hspace{-1mm}>)\|^2\|\nabla \bv_{j,t}(t_5^{*})\|^2\Big).\label{phibd2}
\end{align}
Applying similar techniques to \eqref{psibd}, we get
\begin{align}
    \frac{1}{2\Delta t}\lp\|\bpsi_{j,h}^{n+1}\|^2-\|\bpsi_{j,h}^{n}\|^2\rp+\frac{1}{4\Delta t}\|\bpsi_{j,h}^{n+1}-\bpsi_{j,h}^{n}\|^2+\frac{\Bar{\nu}+\Bar{\nu}_{m}}{4}\|\nabla \bpsi_{j,h}^{n+1}\|^2\nonumber\\+\frac{2\mu-1}{4}\Delta t\|l^{n}_{v,h}\nabla \bpsi_{j,h}^{n+1}\|^2\le\frac{|\nu_j-\nu_{m,j}|}{4}\|\nabla\bphi_{j,h}^n\|^2+\frac{|\nu_j^{'}+\nu_{m,j}^{'}|}{4}\|\nabla\bpsi_{j,h}^n\|^2\nonumber\\+\frac{11(\nu_j-\nu_{m,j})^2}{4\alpha_j}\|\nabla \bfeta_{\bv,j}^n\|^2+\frac{11(\Bar{\nu}+\Bar{\nu}_{m})^2}{4\alpha_j}\|\nabla \bfeta_{\bw,j}^{n+1}\|^2+\frac{11(\nu_j^{'}+\nu_{m,j}^{'})^2}{4\alpha_j}\|\nabla \bfeta_{\bw,j}^n\|^2\nonumber\\+C\frac{\mu^2\Delta t^2}{\alpha_j}\|l^{n}_{v,h}\|_{L^4}^4+\frac{4\mu^2\Delta t}{2\mu-1}\|l^{n}_{v,h}\nabla \bfeta_{\bw,j}^{n+1}\|^2+\frac{C}{\alpha_j}\Delta t^2\|\hspace{-1mm}<\hspace{-1mm}\bphi_{h}\hspace{-1mm}>^n\hspace{-1mm}\|^2\|\bw_{j,t}(s^{*})\|_{\bH^2}^2\nonumber\\+\frac{C}{\alpha_j}\|\bphi_{j,h}^n\|^2\| \bw_j(t^n)\|^2_{\bH^2}+\frac{C}{\alpha_j}\Delta t^2\|\nabla\hspace{-1mm} <\hspace{-1mm}\bfeta_{\bv}\hspace{-1mm}>^n\|^2\|\nabla \bw_{j,t}(s^{**})\|^2+\frac{C}{\alpha_j}\|\nabla\hspace{-1mm}<\hspace{-1mm}\bv_h\hspace{-1mm}>^n\hspace{-1mm}\|^2\|\nabla \bfeta_{\bw,j}^{n+1}\|^2\nonumber\\+\frac{C}{\alpha_j}\|\nabla \bv_{j,h}^{'n}\|^2\|\nabla\bfeta_{\bw,j} ^n\|^2+\frac{C}{\alpha_j}\|\nabla \bfeta^n_{\bv,j}\|^2\|\nabla \bw_j(t^n)\|^2+C\Delta t^2\Big(\|\bw_{j,tt}(s_1^*)\|^2+\|\nabla \bw_{j,t}(s_2^*)\|^2\nonumber\\+\|\nabla \bv_{j,t}(s_3^{*})\|^2\|\nabla \bw_j(t^{n+1})\|^2+\|\nabla  \bv_{j,t}(s_4^{*})\|^2+\|\nabla \big(\bv_j(t^n)-<\hspace{-1mm}\bv(t^n)\hspace{-1mm}>\hspace{-1mm}\big)\|^2\|\nabla \bw_{j,t}(s_5^{*})\|^2\Big),\label{psibd2}
\end{align}
with $s_1^*,s_2^*,s_3^*,s_4^*,s_5^*,\in [t^n,t^{n+1}]$. Adding \eqref{phibd2} and \eqref{psibd2}, assuming $\mu>1/2$, dropping non-negative terms from left, multiplying both sides  by $2\Delta t$, using regularity assumptions, $\|\bphi_{j,h}^0\|=\|\bpsi_{j,h}^0\|=\|\nabla\bphi_{j,h}^0\|=\|\nabla\bpsi^0_{j,h}\|=0$, $\Delta tM=T$, and sum over the time-steps to find

\begin{align}
    \|\bphi_{j,h}^{M}\|^2+\|\bpsi_{j,h}^{M}\|^2+\frac{\Bar{\nu}+\Bar{\nu}_{m}}{2}\Delta t\left(\|\nabla \bphi_{j,h}^{M}\|^2+\|\nabla \bpsi_{j,h}^{M}\|^2\right)\nonumber\\+\frac{\alpha_j\Delta t}{2}\sum_{n=1}^{M-1}\Big(\|\nabla \bphi_{j,h}^{n}\|^2+\|\nabla \bpsi_{j,h}^{n}\|^2\Big)\le C\Delta t\frac{\mu^2\Delta t^2}{\alpha_j}\sum_{n=0}^{M-1}\Big(\|l^{n}_{v,h}\|_{L^4}^4+\|l^{n}_{w,h}\|_{L^4}^4\Big)\nonumber\\+\frac{8\mu^2\Delta t^2}{2\mu-1}\sum_{n=0}^{M-1}\Big(\|l^{n}_{v,h}\nabla \bfeta_{\bw,j}^{n+1}\|^2+\|l^{n}_{w,h}\nabla \bfeta_{\bv,j}^{n+1}\|^2\Big)\nonumber\\+\frac{C}{\alpha_j}\Delta t^2\sum_{n=1}^{M-1}\Delta t\Big(\|\hspace{-1mm}<\hspace{-1mm}\bphi_{h}\hspace{-1mm}>^n\hspace{-1mm}\|^2\|\bw_{j,t}(t)\|_{L^\infty(0,T;\bH^2(\Omega))}^2+\|\hspace{-1mm}<\hspace{-1mm}\bpsi_{h}\hspace{-1mm}>^n\hspace{-1mm}\|^2\|\bv_{j,t}(t)\|_{L^\infty(0,T;\bH^2(\Omega))}^2\Big)\nonumber\\+\frac{C\Delta t}{\alpha_j}\sum_{n=1}^{M-1}\left(\|\bphi_{j,h}^n\|^2\| \bw_j(t)\|^2_{L^\infty(0,T;\bH^2(\Omega))}+\|\bpsi_{j,h}^n\|^2\| \bv_j(t)\|^2_{L^\infty(0,T;\bH^2(\Omega))}\right)\nonumber\\+\frac{C}{\alpha_j}\sum_{n=0}^{M-1}\Delta t\Big(\|\nabla\hspace{-1mm}<\hspace{-1mm}\bv_h\hspace{-1mm}>^n\hspace{-1mm}\|^2\|\nabla \bfeta_{\bw,j}^{n+1}\|^2+\|\nabla\hspace{-1mm}<\hspace{-1mm}\bw_h\hspace{-1mm}>^n\hspace{-1mm}\|^2\|\nabla \bfeta_{\bv,j}^{n+1}\|^2\Big)\nonumber\\+\frac{C}{\alpha_j}\sum_{n=0}^{M-1}\Delta t\left(\|\nabla \bv_{j,h}^{'n}\|^2\|\nabla\bfeta_{\bw,j} ^n\|^2+\|\nabla \bw_{j,h}^{'n}\|^2\|\nabla\bfeta_{\bv,j} ^n\|^2\right)+CT\lp\Delta t^2+\frac{h^{2k}}{\alpha_j}+\frac{h^{2k}\Delta t^2}{\alpha_j}\rp.\label{regularity}
\end{align}

For the first sum on the right-hand-side, we get different bounds for 2D and 3D due to different
Sobolev embedding:

\begin{align*}
    2D:\hspace{3mm} \|l^{n}_{v,h}\|_{L^4}^4&\le C\max_j\|\bv_{j,h}^{'n}\|^2\|\nabla \bv_{j,h}^{'n}\|^2\le C\max_j\|\nabla \bv_{j,h}^{'n}\|^2,\\
    3D:\hspace{3mm} \|l^{n}_{v,h}\|_{L^4}^4&\le C\max_j\|\bv_{j,h}^{'n}\|\|\nabla \bv_{j,h}^{'n}\|^3\le C\max_j\|\nabla \bv_{j,h}^{'n}\|^3,
\end{align*}
where the second upper bound in each inequality coming from the stability theorem, and similarly for $\bw_{j,h}^{'n}$. With the inverse inequality and the stability bound (used on the $L^
2$ norm), we obtain
\begin{align*}
    \|\nabla \bv_{j,h}^{'n}\|\le Ch^{-1}\|\bv_{j,h}^{'n}\|\le Ch^{-1}.
\end{align*}
Thus, the bounds for both 2D or 3D:
\begin{align*}
    \|l^{n}_{v,h}\|_{L^4}^4&\le Ch^{2-d}\max_j\|\nabla \bv_{j,h}^{'n}\|^2,\\
    \|l^{n}_{w,h}\|_{L^4}^4&\le Ch^{2-d}\max_j\|\nabla \bw_{j,h}^{'n}\|^2.
\end{align*}
Using these bounds and the stability bound, the first sum on the right is bounded as

\begin{align*}
    C\Delta t\frac{\mu^2\Delta t^2}{\alpha_j}\sum_{n=0}^{M-1}\left(\|l^{n}_{v,h}\|_{L^4}^4+\|l^{n}_{w,h}\|_{L^4}^4\right)\le& Ch^{2-d}\Delta t\frac{\mu^2\Delta t^2}{\alpha_j}\max_j\left(\|\nabla \bv_{j,h}^{'n}\|^2+\|\nabla \bw_{j,h}^{'n}\|^2\right)\\\le& Ch^{2-d}\frac{\mu^2\Delta t^2}{\alpha_j}.
\end{align*}

For the first part (the second part follows analogously) of the second sum on the right in \eqref{regularity}, we use Agmon's inequality \cite{Robinson2016Three-Dimensional}, the inverse inequality \cite{BS08}, standard estimates of the $L^2$ projection error in the $H^1$ norm for the finite element functions, and the stability estimate \textcolor{black}{to obtain}

\begin{align*}
    \Delta t^2\sum_{n=0}^{M-1}\|l^{n}_{v,h}\nabla \bfeta_{\bw,j}^{n+1}\|^2&\le\Delta t^2\sum_{n=0}^{M-1}\|l^{n}_{v,h}\|_{\infty}^2\|\nabla \bfeta_{\bw,j}^{n+1}\|^2\\&\le Ch^{-1}\Delta t^2\sum_{n=0}^{M-1}\lp\max_j\|\nabla \bv_{j,h}^{'n}\|^2\rp\|\nabla \bfeta_{\bw,j}^{n+1}\|^2\\&\le Ch^{2k-1}\Delta t^2\sum_{n=0}^{M-1}\lp\max_j\|\nabla \bv_{j,h}^{'n}\|^2\rp| \bw_j^{n+1}|_{k+1}^2\\
    &\le Ch^{2k-1}\Delta t.
\end{align*}
Using the above bounds, stability estimate, and standard bounds for $\|\nabla\bfeta_{\bv,j}\|$ and $\|\nabla\bfeta_{\bw,j}\|$, we have
\begin{align}
    \|&\bphi_{j,h}^{M}\|^2+\|\bpsi_{j,h}^{M}\|^2+\frac{\Bar{\nu}+\Bar{\nu}_{m}}{2}\Delta t\left(\|\nabla \bphi_{j,h}^{M}\|^2+\|\nabla \bpsi_{j,h}^{M}\|^2\right)+\frac{\alpha_j\Delta t}{2}\sum_{n=1}^{M-1}\Big(\|\nabla \bphi_{j,h}^{n}\|^2+\|\nabla \bpsi_{j,h}^{n}\|^2\Big)\nonumber\\ &\le\frac{C}{\alpha_j}\Delta t^2\sum_{n=1}^{M-1}\Delta t\Big(\|\hspace{-1mm}<\hspace{-1mm}\bphi_{h}\hspace{-1mm}>^n\hspace{-1mm}\|^2\|\bw_{j,t}(t)\|_{L^\infty(0,T;\bH^2(\Omega))}^2+\|\hspace{-1mm}<\hspace{-1mm}\bpsi_{h}\hspace{-1mm}>^n\hspace{-1mm}\|^2\|\bv_{j,t}(t)\|_{L^\infty(0,T;\bH^2(\Omega))}^2\Big)\nonumber\\&+\frac{C\Delta t}{\alpha_j}\sum_{n=1}^{M-1}\left(\|\bphi_{j,h}^n\|^2\| \bw_j(t)\|^2_{L^\infty(0,T;\bH^2(\Omega))}+\|\bpsi_{j,h}^n\|^2\| \bv_j(t)\|^2_{L^\infty(0,T;\bH^2(\Omega))}\right)\nonumber\\&+C\big(\Delta t^2+h^{2k}+h^{2k}\Delta t^2+h^{2-d}\Delta t^2+ h^{2k-1}\Delta t\big).\label{regularity1}
\end{align}
 Sum over $j=1,\cdots\hspace{-0.35mm},J$, and apply triangle and Young's inequalities, to get
\begin{align}
    \sum_{j=1}^J\|\bphi_{j,h}^{M}\|^2&+\sum_{j=1}^J\|\bpsi_{j,h}^{M}\|^2+\frac{\alpha_j\Delta t}{2}\sum_{n=1}^{M}\sum_{j=1}^J\Big(\|\nabla \bphi_{j,h}^{n}\|^2+\|\nabla \bpsi_{j,h}^{n}\|^2\Big) \nonumber\\&\le\frac{C}{\alpha_j}\sum_{n=1}^{M-1}\lp\frac{\Delta t^3}{J}+\Delta t\rp\left(\sum_{j=1}^J\|\bphi_{j,h}^n\|^2+\sum_{j=1}^J\|\bpsi_{j,h}^n\|^2\right)\nonumber\\&+CJ\big(\Delta t^2+h^{2k}+h^{2k}\Delta t^2+h^{2-d}\Delta t^2+ h^{2k-1}\Delta t\big).\label{regularity2}
\end{align}

Applying the discrete Gr\"onwall Lemma \ref{dgl}, we have
\begin{align}
    \sum_{j=1}^J\|\bphi_{j,h}^{M}\|^2+\sum_{j=1}^J\|\bpsi_{j,h}^{M}\|^2+\frac{\alpha_j\Delta t}{2}\sum_{n=1}^{M}\sum_{j=1}^J\Big(\|\nabla \bphi_{j,h}^{n}\|^2+\|\nabla \bpsi_{j,h}^{n}\|^2\Big)\nonumber\\\le  \exp\left(\frac{CT}{\alpha_j}\left(1+\frac{\Delta t^2}{J}\right)\right)\lp\Delta t^2+h^{2k}+h^{2k}\Delta t^2+h^{2-d}\Delta t^2+ h^{2k-1}\Delta t\rp.\label{regularity3}
\end{align}

Now using the triangle and inequality we can write
\begin{align}
    \sum_{j=1}^J\|\be_{\bv,j}^{M}\|^2+\sum_{j=1}^J\|\be_{\bw,j}^{M}\|^2+\frac{\alpha_j\Delta t}{2}\sum_{n=1}^{M}\sum_{j=1}^J\Big(\|\nabla \be_{\bv,j}^{n}\|^2+\|\nabla \be_{\bw,j}^{n}\|^2\Big)\nonumber\\\le 2\Big(\sum_{j=1}^J\|\bphi_{j,h}^{M}\|^2+\sum_{j=1}^J\|\bfeta_{\bv,j}^{M}\|^2+\sum_{j=1}^J\|\bpsi_{j,h}^{M}\|^2+\sum_{j=1}^J\|\bfeta_{\bw,j}^{M}\|^2\nonumber\\+\frac{\alpha_j\Delta t}{2}\sum_{n=1}^{M}\sum_{j=1}^J\lp\|\nabla \bphi_{j,h}^{n}\|^2+\|\nabla \bfeta_{\bv,j}^{n}\|^2+\|\nabla \bpsi_{j,h}^{n}\|^2+\|\nabla \bfeta_{\bw,j}^{n}\|^2\rp\Big)\nonumber\\\le  C\exp\left(\frac{CT}{\alpha_j}\left(1+\frac{\Delta t^2}{J}\right)\right)\big(\Delta t^2+h^{2k}+h^{2k}\Delta t^2+h^{2-d}\Delta t^2+ h^{2k-1}\Delta t+h^{2k+2}\big).
\end{align}

Finally, again the use of triangle and Young's inequality completes the proof.
\end{proof}

\section{Numerical experiments} \label{numerical-experiments}
As the proposed algorithm is decoupled, at each time-step, we have two Oseen-type problems for each of the $J$ realizations. For MHD simulation, the pointwise enforcement of the solenoidal constraint is crucial \cite{HMX17}. In this paper, for all numerical experiments, we use \textcolor{black}{stable} $(P_2,P_1^{disc})$ Scott-Vogelius elements on barycenter refined regular triangular meshes for each of the Oseen-type problem \textcolor{black}{\cite{arnold1992quadratic}}. \textcolor{black}{The Scott-Vogelius element is pointwise divergence-free and thus allows to enforce the continuity equations $\nabla\cdot\bu_j=0$ and the solenoidal constraints $\nabla\cdot\bB_j=0$ in the discrete level\textcolor{black}{, up to} round-off error.} \textcolor{black}{Thus, we approximate the Els{\"{a}}sser variables $\bv_j$, and $\bw_j$ with a quadratic finite element and $q_j$, and $r_j$ with a linear finite element solving problem with the proposed scheme \eqref{weaknew1}-\eqref{weaknew2}.} We consider the tuning parameter $\mu=1$, \textcolor{black}{number of realizations $J=20$, and the index $j=1,2,\cdots\hspace{-0.35mm},J$} in all experiments. We write the codes, draw the geometries, and generate the regular triangular meshes in Freefem++\cite{H12}. In the first experiment, we \textcolor{black}{test} the predicted convergence rates, while the second experiment shows the energy stability of the scheme, and in \textcolor{black}{the} third and fourth experiments, we show that the scheme performs well in benchmark lid-driven cavity, and channel flow past a step, respectively.

\subsection{Convergence rate verification}  \label{conv-rate} To verify the  spatial and temporal convergence rates, we consider a domain $\Omega=(0,1)^2$ and create structured meshes for $h=1/4$, $1/8$, $1/16$, $1/32$, and $1/64$ using successive refinements. We consider two independent and uniformly distributed random samples for the kinematic viscosity and magnetic diffusivity pair $\{(\nu_j,\nu_{m,j})\in[0.009,0.011]\times[0.09,0.11]\}$, and $\{(\nu_j,\nu_{m,j})\in[0.009,0.011]\times[0.0009,0.0011]\}$ with mean $(\Bar{\nu},\Bar{\nu}_m)=(0.01, 0.1)$, and $(0.01, 0.001)$, respectively. For both samples we consider, the conditions $\alpha_j>0$ hold true. Instead of computing the solution for each pair independently and then taking their average, we compute the average of these $J$ independent solutions by using the proposed ensemble Algorithm \ref{Algn1}.

For this experiment, we begin with the following analytical functions
\begin{align}
{\bv}=\begin{pmatrix}\cos y+(1+e^t)\sin y \\ \sin x+(1+e^t)\cos x\end{pmatrix}, \
{\bw}=\begin{pmatrix}\cos y-(1+e^t)\sin y \\ \sin x-(1+e^t)\cos x\end{pmatrix}, \ p =\sin(x+y)(1+e^t),\hspace{1mm}\text{and}\hspace{1mm} \lambda=0.\label{exact-solution}    
\end{align}
Next, we consider $J$ different manufactured solutions introducing a perturbation parameter $\epsilon$ as
\begin{align}
    \bv_j\textcolor{black}{(x,y,t)}:=\lp 1+\frac{(-1)^{j+1}\lceil j/2\rceil}{5}\epsilon\rp\bv, \hspace{1mm}\text{and}\hspace{1mm}\bw_j\textcolor{black}{(x,y,t)}:=\lp 1+\frac{(-1)^{j+1}\lceil j/2\rceil}{5}\epsilon\rp\bw.\end{align}
The above exact solutions are divergence-free. For each pair $(\nu_j,\nu_{m,j})$, using the above exact solutions, we compute the forcing vectors from \eqref{els1}-\eqref{els2}.  We use $\bv_j^0=\bv_j(x,y,0)$, and $\bw_j^0=\bw_j(x,y,0)$ as the initial conditions and $\bv_j|_{\partial\Omega}=\bv_j$, and $\bw_j|_{\partial\Omega}=\bw_j$ as the boundary conditions.

The ensemble average error is defined as $<\hspace{-1mm}\be_{\bz}\hspace{-1mm}>^n:=<\hspace{-1mm}\bz_h\hspace{-1mm}>^n-<\hspace{-1mm}\bz(t^n)\hspace{-1mm}>$, where $\bz=\bv\hspace{1mm}\text{or}\hspace{1mm}\bw$, which reduces to $<\hspace{-1mm}\be_{\bz}\hspace{-1mm}>^n=<\hspace{-1mm}\bz_h\hspace{-1mm}>^n-\bz(t^n)$. We compute the \textcolor{black}{$L^2(0,T;\bH^1)$} norm of the error and is denoted by $\|\cdot\|_{2,1}$. 

For the spatial convergence, we consider a small end time $T=0.001$ so that the temporal error does not dominate over the spatial error, and use a fixed time-step size $\Delta t=T/8$. We run a complete simulation beginning  with $h=1/4$ and repeat with the successively refined meshes until we have $h=1/64$. In Tables \ref{spatial-convergence-ep-0}-\ref{spatial-convergence-ep-0-01}, we list the norm of the spatial errors and compute the spatial convergence rates for the two sets of samples of the viscosity pair, for several choices of $\epsilon$. In each case, we observe a second order spatial convergence, which is predicted by our error analysis given in Theorem \ref{convergence-theorem}.

\begin{table}[!ht]
	\begin{center}
		\small\begin{tabular}{|c|c|c|c|c|c|c|c|c|}\hline
			\multicolumn{9}{|c|}{Spatial convergence (fixed $T=0.001$, $\Delta t=T/8$) with $j=1,2,\cdots\hspace{-0.35mm}, 20$}\\\hline
		$\hspace{-1.5mm}\epsilon=0.0\hspace{-1.5mm}$	&\multicolumn{4}{|c|}{$\{(\nu_j,\nu_{m,j})\in[0.009,0.011]\times[0.09,0.11]\}\hspace{-1mm}$}&\multicolumn{4}{|c|}{$\hspace{-1mm}\{(\nu_j,\nu_{m,j})\in[0.009,0.011]\times[0.0009,0.0011]\}\hspace{-2mm}$}\\\hline
			$h$ & $\|\hspace{-1mm}<\hspace{-1mm}\be_{\bv}\hspace{-1mm}>\hspace{-1mm}\|_{2,1}$ & rate   &$\|\hspace{-1mm}<\hspace{-1mm}\be_{\bw}\hspace{-1mm}>\hspace{-1mm}\|_{2,1}$  & rate & $\|\hspace{-1mm}<\hspace{-1mm}\be_{\bv}\hspace{-1mm}>\hspace{-1mm}\|_{2,1}$ & rate   &$\|\hspace{-1mm}<\hspace{-1mm}\be_{\bw}\hspace{-1mm}>\hspace{-1mm}\|_{2,1}$  & rate\\ \hline
			 $\frac{1}{4}$ &  1.0741e-04  &&  2.0597e-04 & &  1.0752e-04  &&  2.0602e-04 &\\ \hline
			 $\frac{1}{8}$&  2.7081e-05  &1.99&  5.1620e-05 & 2.00 &  2.7108e-05  &1.99&   5.1653e-05 & 2.00\\\hline
			 $\frac{1}{16}$&  6.802\textcolor{black}{5}e-06  &1.99&  1.295\textcolor{black}{7}e-05 & 1.99 &  6.8037e-06  &1.99&  1.30\textcolor{black}{09}e-05 & 1.99\\\hline
			$\frac{1}{32}$&  1.71\textcolor{black}{68}e-06 &1.99& 3.25\textcolor{black}{23}e-06& 1.99 &  1.7110e-06 &1.99& 3.3181e-06& 1.97\\\hline
			$\frac{1}{64}$ & 4.3\textcolor{black}{049}e-07  & \textcolor{black}{2.00}&  8.1\textcolor{black}{168}e-07  &2.00 & 4.3117e-07  & 1.99&   8.5456e-07  &1.96\\\hline
		\end{tabular}
	\end{center}
	\caption{\footnotesize Spatial errors and  convergence rates for $\bv$ and $\bw$ with $\epsilon=0.0$.}\label{spatial-convergence-ep-0}
\end{table}

\begin{table}[!ht]
	\begin{center}
		\small\begin{tabular}{|c|c|c|c|c|c|c|c|c|}\hline
			\multicolumn{9}{|c|}{Spatial convergence (fixed $T=0.001$, $\Delta t=T/8$) with $j=1,2,\cdots\hspace{-0.35mm}, 20$}\\\hline
		$\hspace{-1.5mm}\epsilon=0.001\hspace{-1.5mm}$	&\multicolumn{4}{|c|}{$\hspace{-1mm}\{(\nu_j,\nu_{m,j})\in[0.009,0.011]\times[0.09,0.11]\}\hspace{-1mm}$}&\multicolumn{4}{|c|}{$\hspace{-1mm}\{(\nu_j,\nu_{m,j})\in[0.009,0.011]\times[0.0009,0.0011]\}\hspace{-2mm}$}\\\hline
			$h$ & $\|\hspace{-1mm}<\hspace{-1mm}\be_{\bv}\hspace{-1mm}>\hspace{-1mm}\|_{2,1}$ & rate   &$\|\hspace{-1mm}<\hspace{-1mm}\be_{\bw}\hspace{-1mm}>\hspace{-1mm}\|_{2,1}$  & rate & $\|\hspace{-1mm}<\hspace{-1mm}\be_{\bv}\hspace{-1mm}>\hspace{-1mm}\|_{2,1}$ & rate   &$\|\hspace{-1mm}<\hspace{-1mm}\be_{\bw}\hspace{-1mm}>\hspace{-1mm}\|_{2,1}$  & rate\\ \hline
			 $\frac{1}{4}$ & 1.0741e-04 && 2.0597e-04   & &1.0752e-04    && 2.0602e-04  &\\ \hline
			 $\frac{1}{8}$& 2.7081e-05  &1.99&  5.1620e-05   & 2.00 &  2.7108e-05 &1.99&  5.1653e-05  & 2.00\\\hline
			 $\frac{1}{16}$& 6.8025e-06  &1.99& 1.2957e-05 & 1.99 &  6.8037e-06 &1.99& 1.3009e-05  & 1.99\\\hline
			$\frac{1}{32}$&  1.7168e-06 &1.99&3.2523e-06 & 1.99 &   1.7110e-06&1.99 & 3.3181e-06&1.97\\\hline
			$\frac{1}{64}$ & 4.3048e-07  &2.00 & 8.1168e-07   &2.00 & 4.3123e-07  &1.99 & 8.5458e-07    &1.96\\\hline
		\end{tabular}
	\end{center}
	\caption{\footnotesize Spatial errors and  convergence rates for $\bv$ and $\bw$ with $\epsilon=0.001$.}\label{spatial-convergence-ep-0-001}
\end{table}

\begin{table}[!ht]
	\begin{center}
		\small\begin{tabular}{|c|c|c|c|c|c|c|c|c|}\hline
			\multicolumn{9}{|c|}{Spatial convergence (fixed $T=0.001$, $\Delta t=T/8$) with $j=1,2,\cdots\hspace{-0.35mm}, 20$}\\\hline
		$\hspace{-1.5mm}\epsilon=0.01\hspace{-1.5mm}$	&\multicolumn{4}{|c|}{$\hspace{-1mm}\{(\nu_j,\nu_{m,j})\in[0.009,0.011]\times[0.09,0.11]\}$}&\multicolumn{4}{|c|}{$\{(\nu_j,\nu_{m,j})\in[0.009,0.011]\times[0.0009,0.0011]\}\hspace{-2mm}$}\\\hline
			$h$ & $\|\hspace{-1mm}<\hspace{-1mm}\be_{\bv}\hspace{-1mm}>\hspace{-1mm}\|_{2,1}$ & rate   &$\|\hspace{-1mm}<\hspace{-1mm}\be_{\bw}\hspace{-1mm}>\hspace{-1mm}\|_{2,1}$  & rate & $\|\hspace{-1mm}<\hspace{-1mm}\be_{\bv}\hspace{-1mm}>\hspace{-1mm}\|_{2,1}$ & rate   &$\|\hspace{-1mm}<\hspace{-1mm}\be_{\bw}\hspace{-1mm}>\hspace{-1mm}\|_{2,1}$  & rate\\ \hline
			 $\frac{1}{4}$ & 1.0741e-04  &&  2.0597e-04   & & 1.0752e-04   && 2.0602e-04  &\\ \hline
			 $\frac{1}{8}$&2.7081e-05   &1.99& 5.1620e-05 &2.00&   2.7108e-05& 1.99 &5.1653e-05  & 2.00\\\hline
			 $\frac{1}{16}$&6.8031e-06 &1.99&1.2957e-05&1.99 & 6.8046e-06  &1.99& 1.3009e-05  & 1.99\\\hline
			$\frac{1}{32}$& 1.7182e-06 &1.99&3.2527e-06&1.99 & 1.7146e-06  &1.99&3.3186e-06 &1.97 \\\hline
			$\frac{1}{64}$ & 4.3552e-07  & 1.98& 8.1418e-07   & 2.00& 4.5043e-07  &1.93 &8.6250e-07     &1.94\\\hline
		\end{tabular}
	\end{center}
	\caption{\footnotesize Spatial errors and  convergence rates for $\bv$ and $\bw$ with $\epsilon=0.01$.}\label{spatial-convergence-ep-0-01}
\end{table}

\begin{table}[!ht]
	\begin{center}
		\small\begin{tabular}{|c|c|c|c|c|c|c|c|c|}\hline
			\multicolumn{9}{|c|}{Temporal convergence (fixed $h=1/64$, $T=1$) with $j=1,2,\cdots\hspace{-0.35mm}, 20$}\\\hline
			$\hspace{-1mm}\epsilon=0.0\hspace{-1mm}$&\multicolumn{4}{|c|}{$\hspace{-1mm}\{(\nu_j,\nu_{m,j})\in[0.009,0.011]\times[0.09,0.11]\}\hspace{-1mm}$}&\multicolumn{4}{|c|}{\hspace{-1mm}$\{(\nu_j,\nu_{m,j})\in[0.009,0.011]\times[0.0009,0.0011]\}\hspace{-2mm}$}\\\hline
			$\Delta t$ & $\|\hspace{-1mm}<\hspace{-1mm}\be_{\bv}\hspace{-1mm}>\hspace{-1mm}\|_{2,1}$ & rate   &$\|\hspace{-1mm}<\hspace{-1mm}\be_{\bw}\hspace{-1mm}>\hspace{-1mm}\|_{2,1}$  & rate & $\|\hspace{-1mm}<\hspace{-1mm}\be_{\bv}\hspace{-1mm}>\hspace{-1mm}\|_{2,1}$ & rate   &$\|\hspace{-1mm}<\hspace{-1mm}\be_{\bw}\hspace{-1mm}>\hspace{-1mm}\|_{2,1}$  & rate \\ \hline
			 $\frac{T}{2}$ &2.2419e-01  & & 1.6653e-01& &9.5227e-01  & & 7.5549e-01&\\ \hline
			 $\frac{T}{4}$& 1.0851e-01  &1.05 &7.9697e-02 & 1.06 & 5.2621e-01  &0.86 &4.5248e-01 & 0.74\\\hline
			 $\frac{T}{8}$& 5.5987e-02  &0.95& 4.0986e-02 & 0.96 & 3.0570e-01  &0.78& 2.7650e-01 & 0.71\\\hline
			 $\frac{T}{16}$&2.9231e-02   &0.94& 2.1480e-02  & 0.93 &1.7541e-01   &0.80& 1.6320e-01  & 0.76\\\hline
			$\frac{T}{32}$& 1.5075e-02   &0.96& 1.1117e-02 & 0.95 & 9.6367e-02   &0.86& 9.0908e-02 & 0.84\\\hline
			$\frac{T}{64}$ & 7.6728e-03   &0.97& 5.6704e-03  & 0.97 & 5.0830e-02    &0.92& 4.8256e-02  & 0.91\\\hline
			$\frac{T}{128}$ & 3.8730e-03  &0.99& 2.8659e-03  & 0.98 & 2.6137e-02  &0.96& 2.4887e-02  & 0.96\\\hline
		\end{tabular}
	\end{center}
	\caption{\footnotesize Temporal errors and  convergence rates for $\bv$ and $\bw$ with $\epsilon=0.0.$}\label{temporal-convergence-ep-0-0}
\end{table}

\begin{table}[!ht]
	\begin{center}
		\small\begin{tabular}{|c|c|c|c|c|c|c|c|c|}\hline
			\multicolumn{9}{|c|}{Temporal convergence (fixed $h=1/64$, $T=1$) with $j=1,2,\cdots\hspace{-0.35mm}, 20$}\\\hline
			$\hspace{-1.5mm}\epsilon=0.001\hspace{-1.5mm}$&\multicolumn{4}{|c|}{$\hspace{-1mm}\{(\nu_j,\nu_{m,j})\in[0.009,0.011]\times[0.09,0.11]\}\hspace{-1mm}$}&\multicolumn{4}{|c|}{$\hspace{-1mm}\{(\nu_j,\nu_{m,j})\in[0.009,0.011]\times[0.0009,0.0011]\}\hspace{-2mm}$}\\\hline
			$\Delta t$ & $\|\hspace{-1mm}<\hspace{-1mm}\be_{\bv}\hspace{-1mm}>\hspace{-1mm}\|_{2,1}$ & rate   &$\|\hspace{-1mm}<\hspace{-1mm}\be_{\bw}\hspace{-1mm}>\hspace{-1mm}\|_{2,1}$  & rate & $\|\hspace{-1mm}<\hspace{-1mm}\be_{\bv}\hspace{-1mm}>\hspace{-1mm}\|_{2,1}$ & rate   &$\|\hspace{-1mm}<\hspace{-1mm}\be_{\bw}\hspace{-1mm}>\hspace{-1mm}\|_{2,1}$  & rate \\ \hline
			 $\frac{T}{2}$ & 2.2417e-01  & &1.6646e-01 & &9.5155e-01   & &7.5403e-01 &\\ \hline
			 $\frac{T}{4}$& 1.0850e-01  &1.05 &7.9669e-02 &1.06  & 5.2583e-01  &0.86 &4.5182e-01 & 0.74 \\\hline
			 $\frac{T}{8}$& 5.5981e-02  &0.95& 4.0972e-02 & 0.96 & 3.0551e-01 &0.78& 2.7619e-01 &0.71 \\\hline
			 $\frac{T}{16}$&2.9228e-02  &0.94& 2.1473e-02  &0.93  & 1.7532e-01  &0.80& 1.6307e-01  &0.76 \\\hline
			$\frac{T}{32}$& 1.5073e-02   &0.96& 1.1114e-02 &0.95  & 9.6320e-02  &0.86& 9.0848e-02 &0.84\\\hline
			$\frac{T}{64}$ & 7.6711e-03 &0.97& 5.6682e-03 & 0.97 & 5.0802e-02 &0.92& 4.8223e-02 &0.91 \\\hline
			$\frac{T}{128}$ & 3.8715e-03  &0.99& 2.8645e-03  & 0.98&  2.6115e-02& 0.96 & 2.4862e-02 &0.96\\\hline
		\end{tabular}
	\end{center}
	\caption{\footnotesize Temporal errors and  convergence rates for $\bv$ and $\bw$ with $\epsilon=0.001.$}
\end{table}

\begin{table}[!ht]
	\begin{center}
		\small\begin{tabular}{|c|c|c|c|c|c|c|c|c|}\hline
			\multicolumn{9}{|c|}{Temporal convergence (fixed $h=1/64$, $T=1$) with $j=1,2,\cdots\hspace{-0.35mm}, 20$}\\\hline
			$\hspace{-1.5mm}\epsilon=0.01\hspace{-1.5mm}$&\multicolumn{4}{|c|}{$\hspace{-1mm}\{(\nu_j,\nu_{m,j})\in[0.009,0.011]\times[0.09,0.11]\}\hspace{-1mm}$}&\multicolumn{4}{|c|}{$\hspace{-1mm}\{(\nu_j,\nu_{m,j})\in[0.009,0.011]\times[0.0009,0.0011]\}\hspace{-2mm}$}\\\hline
			$\Delta t$ & $\|\hspace{-1mm}<\hspace{-1mm}\be_{\bv}\hspace{-1mm}>\hspace{-1mm}\|_{2,1}$ & rate   &$\|\hspace{-1mm}<\hspace{-1mm}\be_{\bw}\hspace{-1mm}>\hspace{-1mm}\|_{2,1}$  & rate & $\|\hspace{-1mm}<\hspace{-1mm}\be_{\bv}\hspace{-1mm}>\hspace{-1mm}\|_{2,1}$ & rate   &$\|\hspace{-1mm}<\hspace{-1mm}\be_{\bw}\hspace{-1mm}>\hspace{-1mm}\|_{2,1}$  & rate \\ \hline
			 $\frac{T}{2}$ &2.2247e-01  & &1.6093e-01 & & 8.9271e-01 & &6.4701e-01 &\\ \hline
			 $\frac{T}{4}$& 1.0787e-01  & 1.04&7.7193e-02 &1.06  & 4.9474e-01  &0.85 &3.9981e-01 &0.69 \\\hline
			 $\frac{T}{8}$& 5.5785e-02  &0.95& 3.9847e-02 &0.95  & 2.8963e-01 &0.77& 2.5194e-01 &0.67 \\\hline
			 $\frac{T}{16}$&2.9250e-02  &0.93& 2.1008e-02  &0.92  & 1.6848e-01 &0.78& 1.5314e-01  &0.72 \\\hline
			$\frac{T}{32}$& 1.5207e-02   &0.94& 1.0976e-02 & 0.94 & 9.4351e-2  &0.84& 8.7772e-02 &0.80\\\hline
			$\frac{T}{64}$ & 7.8585e-03 &0.95& 5.6893e-03 & 0.95 & 5.0981e-02 &0.89& 4.7969e-02 &0.87 \\\hline
			$\frac{T}{128}$ & 4.0843e-03  &0.94& 2.9624e-03  &0.94 &2.7116e-02 &0.91& 2.5650e-02 &0.90 \\\hline
		\end{tabular}
	\end{center}
	\caption{\footnotesize Temporal errors and  convergence rates for $\bv$ and $\bw$ with $\epsilon=0.01.$}\label{temporal-convergence-ep-0-01}
\end{table}

To observe the temporal convergence rates, we use a fixed $h=1/64$, and the simulation end time $T=1$ and run  the simulations varying the time-step size as $\Delta t=T/2,\hspace{1mm}T/4,\hspace{1mm}T/8,\hspace{1mm}T/16,\hspace{1mm}T/32, T/64$, and $T/128$.
In Tables \ref{temporal-convergence-ep-0-0}-\ref{temporal-convergence-ep-0-01}, we represent the $L^2(0,T;\bH^1)$ norm of the temporal errors and their convergence rates. As $\Delta t\rightarrow 0$, we observe a first order temporal convergence rate, which is also consistent with the theoretical analysis in Theorem \ref{convergence-theorem}.

\vspace{-1ex}
\subsection{Energy stability test}  To test the energy stability of the proposed scheme \eqref{weaknew1}-\eqref{weaknew2}, we keep the same domain, the initial conditions with $\epsilon=0.01$, and the finite element pairs as given in the previous experiment \ref{conv-rate}. A uniformly distributed random of sample $\{(\nu_j,\nu_{m,j})\in[0.009,0.011]\times[0.09,0.11]\}$ with mean $(\Bar{\nu},\Bar{\nu}_m)=(0.01,0.1)$ is considered, so that $\alpha_j>0$, for all $j$. Clearly, the sample has a maximum $10\%$ fluctuation from the mean.  We consider homogeneous boundary conditions for the velocity and magnetic field, and zero body forces (i.e. $\bif_{j}=\bg_{j}=\textbf{0}$) so that the system does not have any external source of energy. We choose $h=1/32$, time-step size $\Delta t=0.05$, and solve the problem in \eqref{els1}-\eqref{els2} by the proposed Algorithm \ref{Algn1}. We define the energy of the system as:
\textcolor{black}{$$E_h^n:=\frac12\left(\|\hspace{-1mm} <\hspace{-1mm}\bv_h\hspace{-1mm}>^n\hspace{-1mm}\|^2+\|\hspace{-1mm}<\hspace{-1mm} \bw_h^n\hspace{-1mm}>\hspace{-1mm}\|^2\right).$$} 
\noindent The time evolution of energy until the end time $T=1$ is showing in Fig. \ref{energy-curve}. We observe that the system is energy dissipation, and is consistent with the stability result in Theorem \ref{stability-analysis}. 
\begin{figure}[h!]
\begin{center}
            \includegraphics[width = 0.6\textwidth, height=0.35\textwidth,viewport=-300 0 1200 890, clip]{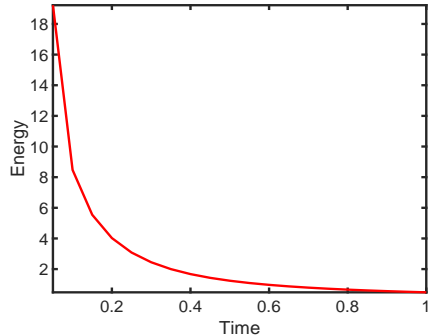}
\end{center}\caption{Energy of the MHD system versus time.}\label{energy-curve}
\end{figure}

\subsection{Lid-driven cavity}
In this test, we consider a 2D benchmark regularized lid-driven cavity problem \cite{balajewicz2013low, fick2018stabilized, lee2019study} with domain \textcolor{black}{$\Omega=(-1,1)^2$}. No-slip boundary condition for \textcolor{black}{the} velocity is \textcolor{black}{enforced on} all boundaries except the top (which is the lid of the cavity), where we impose $$\bu_j=\textcolor{black}{\lp 1+\frac{(-1)^{j+1}\lceil j/2\rceil}{5}\epsilon\rp}\begin{pmatrix}(1-x^2)^2\\0\end{pmatrix}.$$ For the magnetic field boundary conditions, we assign $$\bB_j=\textcolor{black}{\lp 1+\frac{(-1)^{j+1}\lceil j/2\rceil}{5}\epsilon\rp}\begin{pmatrix}0\\1\end{pmatrix}$$ on \textcolor{black}{all sides}. We assume the flow begins from rest, initially there is no magnetic field, and no external source is present in the system (i.e. $\bif_{j}=\bg_{j}=\textbf{0}$). We \textcolor{black}{generate a computational} mesh that provides a total of 1,307,690 \textcolor{black}{degrees of freedom (dofs)} for each of velocity and magnetic field and a total of $163,702$ \textcolor{black}{dofs} for each of pressure and magnetic pressure.

To study the long-time unsteady flow behavior,  we \textcolor{black}{first} validate our computation with available data \textcolor{black}{from} the literature \cite{fick2018stabilized}. Thus, we run a simulation in absence of the magnetic field (setting $s=0$ in the model) with the Reynolds number $Re=15000$ (that is, no perturbation in \textcolor{black}{the viscosities, initial and boundary conditions are} considered). We define the viscosity $\nu=2/Re$, since the maximum velocity of the lid is 1 and the characteristic length is 2. Thanks to the unconditional stability, we run the simulation with a large time-step size $\Delta t=5$ until the end $T=600$ and plot the solution in Fig. \ref{ldc_s_0_Re_15000}. We observe a large primary vortex in the center of the cavity, and other vortices are near to the three corners except for the upper right. We note that the same observation was made by Fick et al. in \cite{fick2018stabilized}.

Next, we consider a total of $20$ uniformly distributed random Reynolds numbers and magnetic diffusivities from the intervals $[13636.36, 16666.67]$ and $[0.009,0.011]$, respectively. That is, the sample mean of the Reynolds numbers and the sample mean of the magnetic diffusivities are $15151.52$, and $0.01$, respectively.

We run the simulations for several values of the coupling parameter $s$ \textcolor{black}{with a fixed $\epsilon=0.01$} and plot the velocity (speed contour) and magnetic field (strength) solutions at $T=600$ in Figures  \ref{ldc_u_s_0_001_to_s_1}-\ref{ldc_B_s_0_001_to_s_1}. From the speed contour plots, Fig. \ref{ldc_u_s_0_001_to_s_1}, as $s$ increases, a change in the flow structure is observed and with $s=1$, the center of the circulation gets close to $(0,0)$ and the magnetic field strength \textcolor{black}{realizes} a type of reflection symmetry.

\begin{figure}[h!]
\begin{center}
            \includegraphics[width = 0.4\textwidth, height=0.35\textwidth,clip]{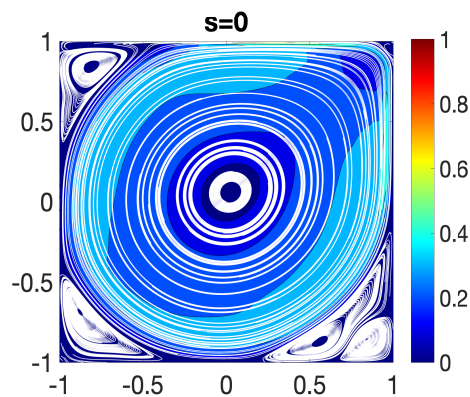}
	 	 	\caption{A lid-driven cavity problem. Velocity solution (shown as streamlines over speed contours) for $Re=15000$. \label{ldc_s_0_Re_15000}}
\end{center}
\end{figure}

\begin{figure}[ht] 
  \begin{subfigure}[b]{0.5\linewidth}
    \centering
    \includegraphics[width=0.8\textwidth, height=0.7\textwidth]{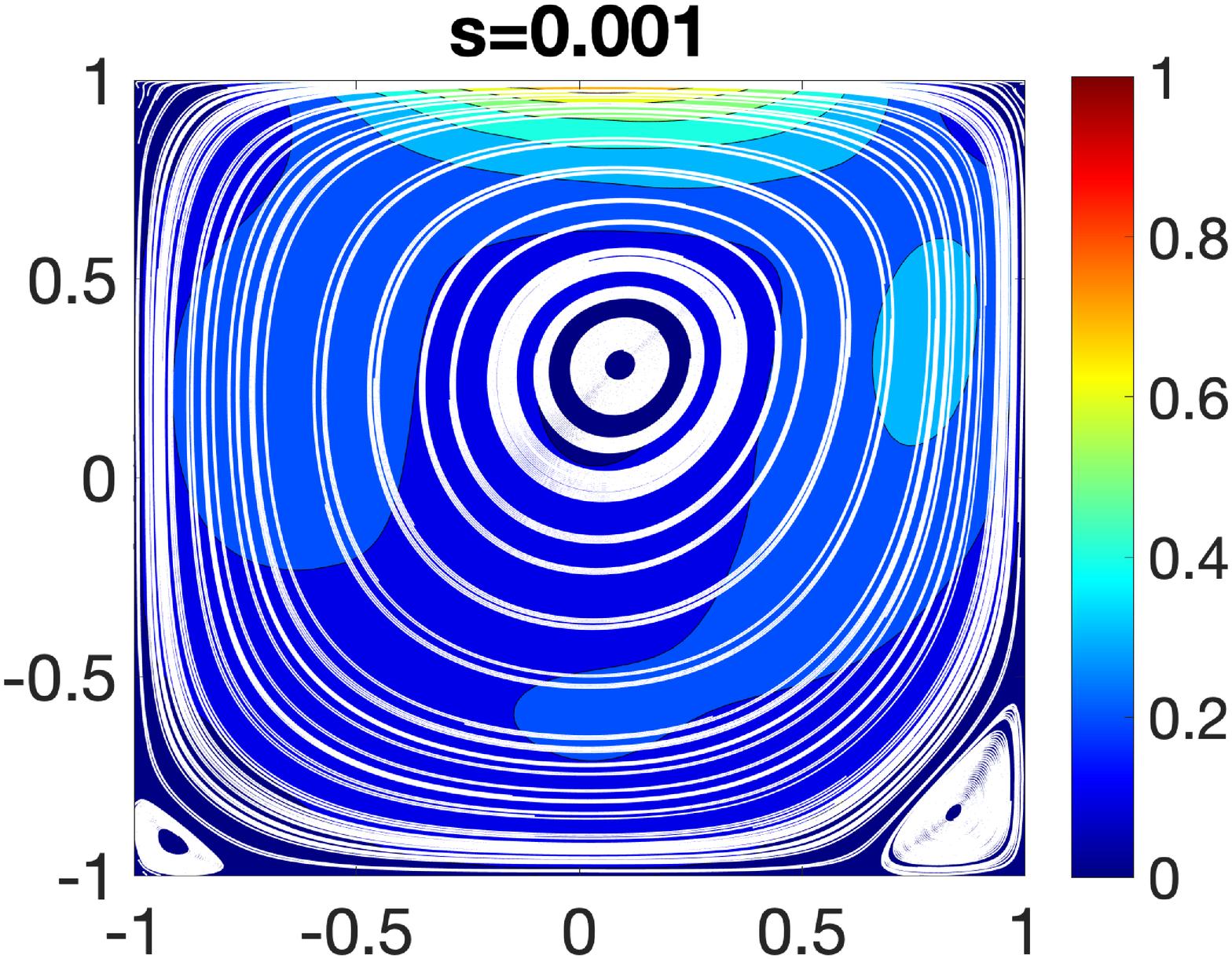}\vspace{-4.5mm}
  \end{subfigure}\hspace{-10ex}
  \begin{subfigure}[b]{0.5\linewidth}
    \centering
    \includegraphics[width=0.8\textwidth, height=0.7\textwidth]{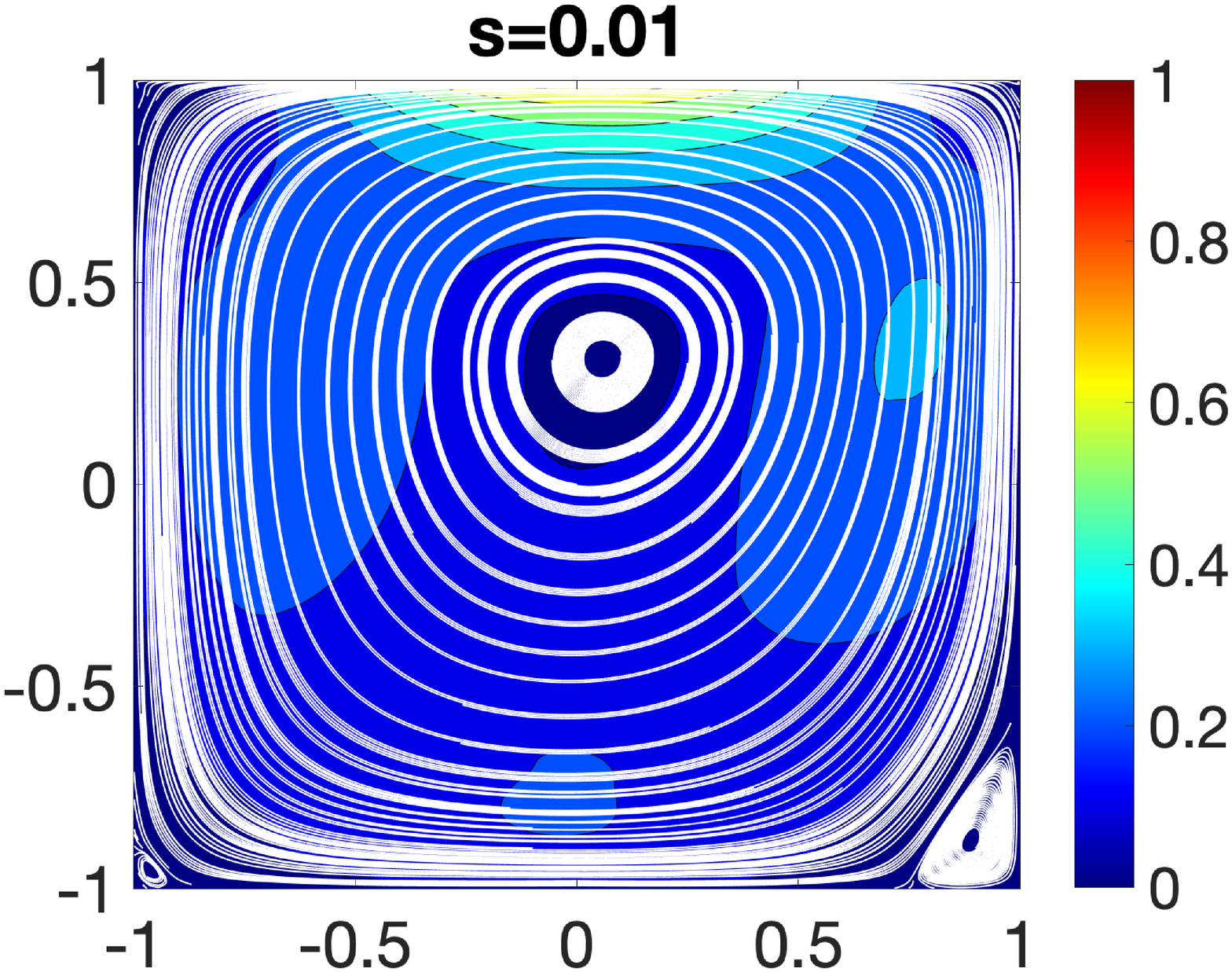}\vspace{-4.5mm}
  \end{subfigure} \\
  \begin{subfigure}[b]{0.5\linewidth}
    \centering
    \includegraphics[width=0.8\textwidth, height=0.7\textwidth]{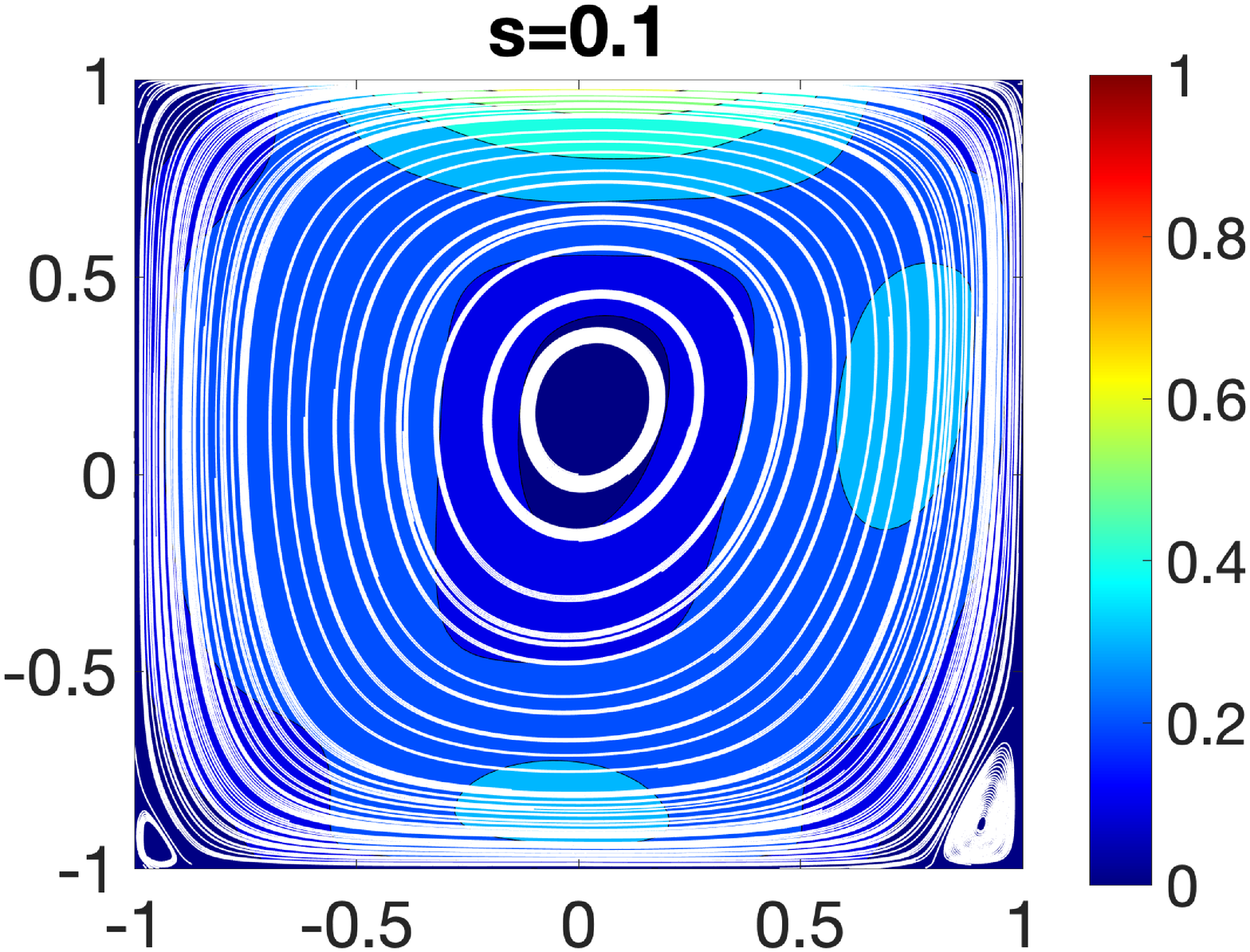} 
  \end{subfigure}\hspace{-10ex}
  \begin{subfigure}[b]{0.5\linewidth}
    \centering
    \includegraphics[width=0.8\textwidth, height=0.7\textwidth]{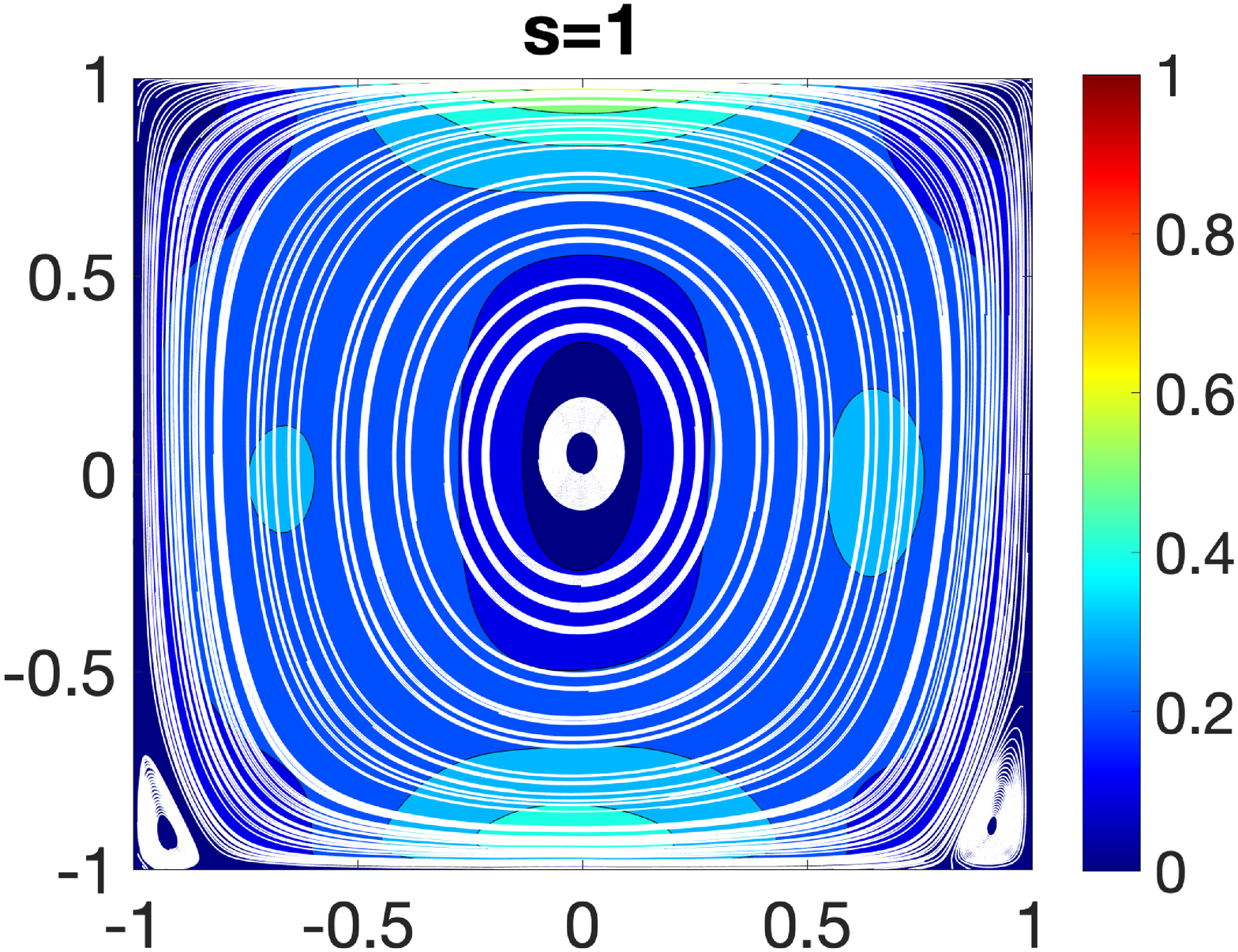} 
  \end{subfigure} 
  \caption{A lid-driven cavity problem. The velocity ensemble \textcolor{black}{average} solutions \textcolor{black}{for} $13636.36\le Re_j\le 16666.67$ at $T=600$ \textcolor{black}{for various coupling parameter $s$}.}
  \label{ldc_u_s_0_001_to_s_1} 
\end{figure}

\begin{figure}[ht] 
  \begin{subfigure}[b]{0.5\linewidth}
    \centering
    \includegraphics[width=0.8\textwidth, height=0.7\textwidth]{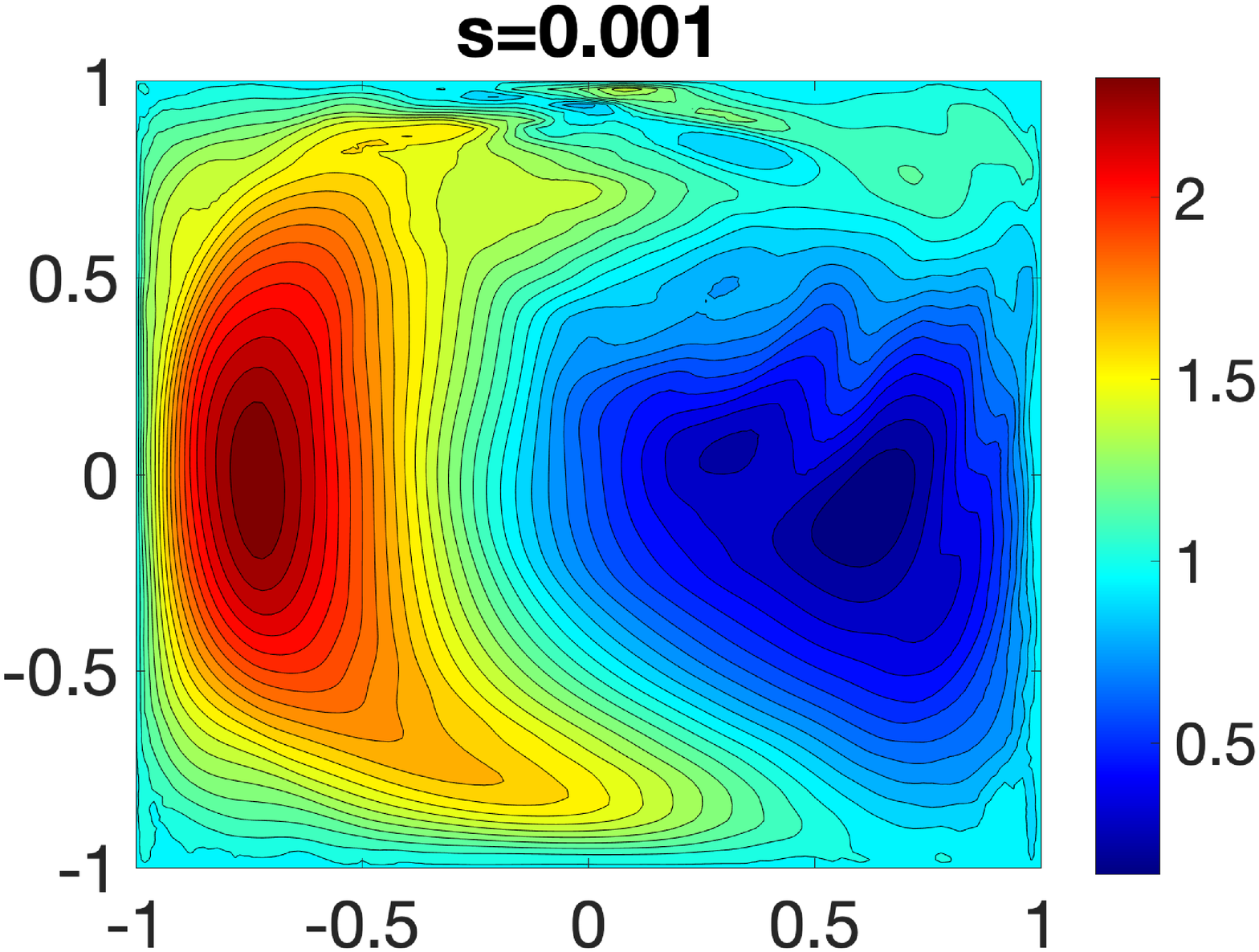}\vspace{-5mm}
  \end{subfigure}\hspace{-10ex}
  \begin{subfigure}[b]{0.5\linewidth}
    \centering
    \includegraphics[width=0.8\textwidth, height=0.7\textwidth]{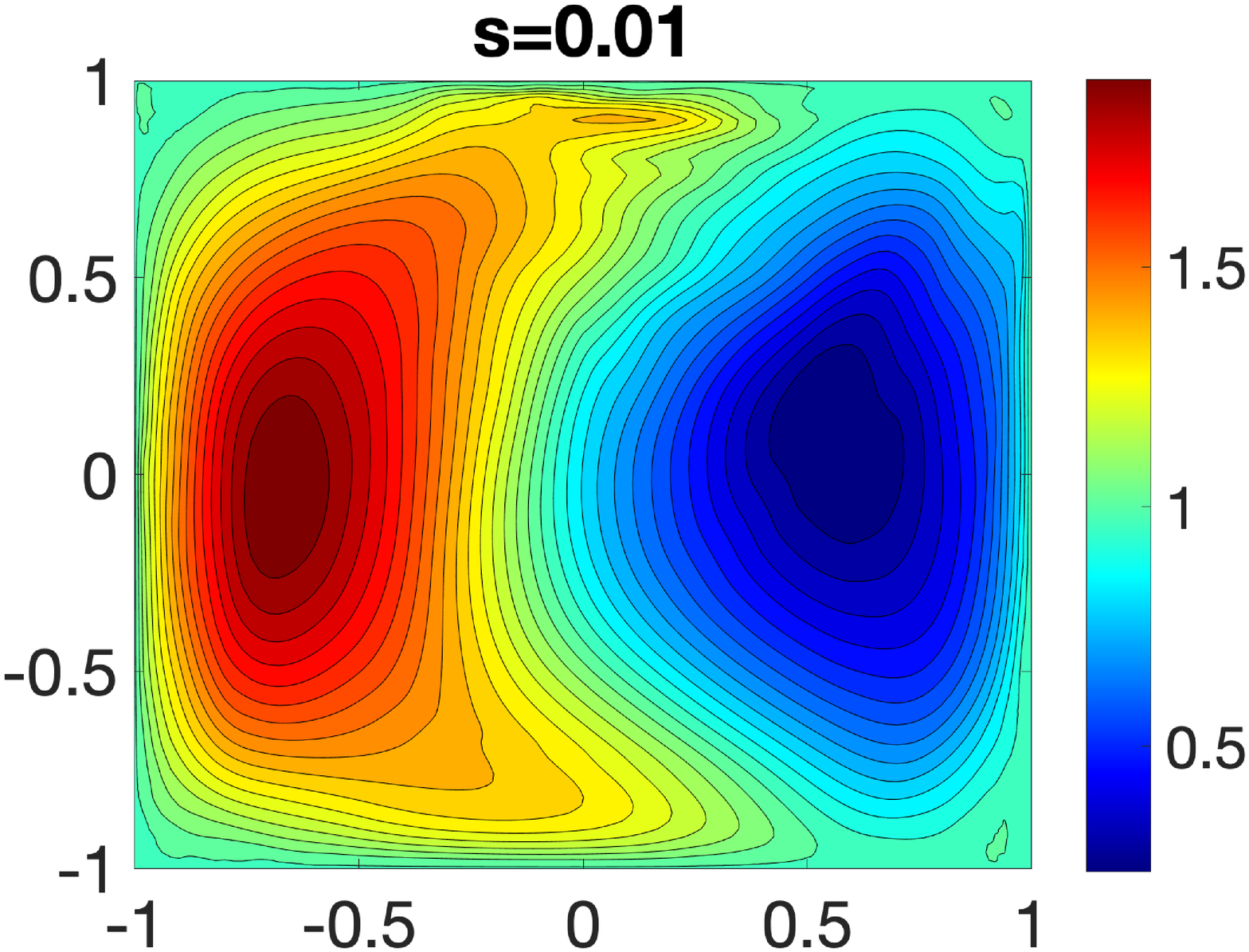}\vspace{-5mm}
  \end{subfigure} \\
  \begin{subfigure}[b]{0.5\linewidth}
    \centering
    \includegraphics[width=0.8\textwidth, height=0.7\textwidth]{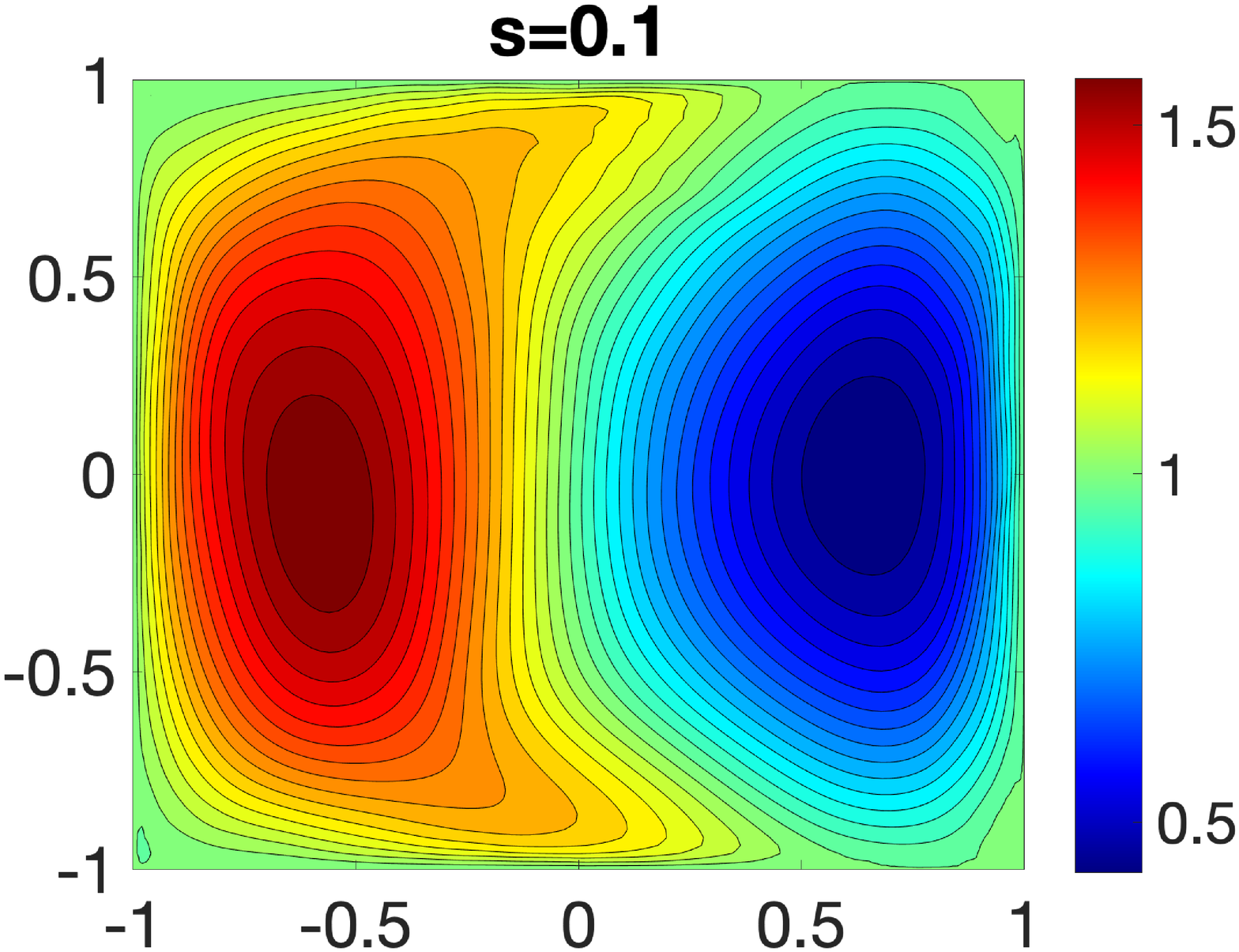} 
  \end{subfigure}\hspace{-10ex}
  \begin{subfigure}[b]{0.5\linewidth}
    \centering
    \includegraphics[width=0.8\textwidth, height=0.7\textwidth]{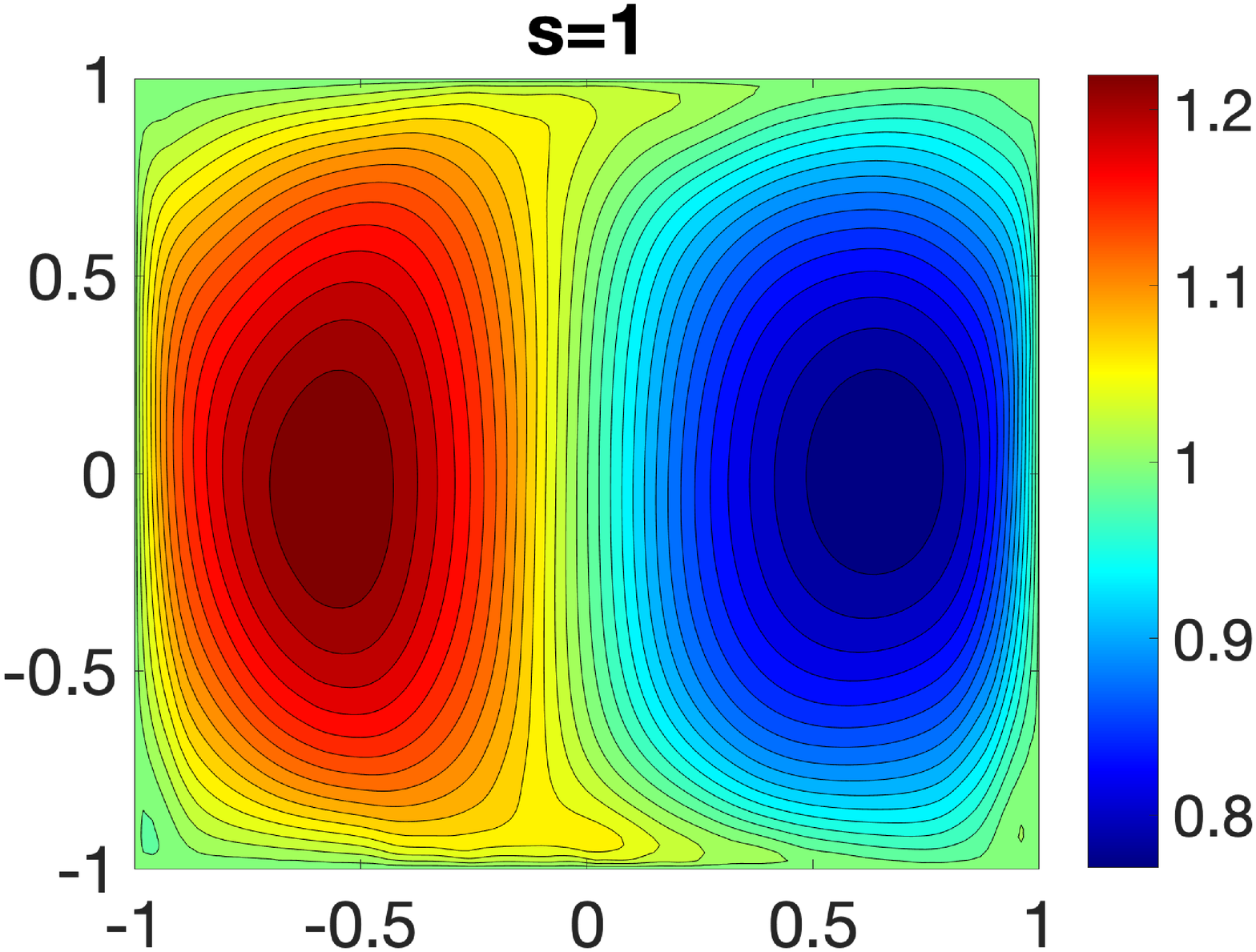} 
  \end{subfigure} 
  \caption{A lid-driven cavity problem. The magnetic field strength \textcolor{black}{ensemble average} solutions \textcolor{red}{for} $13636.36\le Re_j\le 16666.67$ at $T=600$ \textcolor{black}{for various coupling parameter $s$}.}
  \label{ldc_B_s_0_001_to_s_1} 
\end{figure}

\subsection{MHD channel flow over a step}

In this experiment, we consider a benchmark problem \cite{AKMR15,HMR17,layton2008numerical,MR17} in which the domain under consideration is a $30\times 10$ rectangular channel with a $1\times 1$ step at the bottom and five units away from the inlet. \textcolor{black}{The problem is not physically accurate as the domain is not convex, but we run the simulation anyway.} The following initial conditions are chosen:\vspace{-2ex}
\begin{align*}
    \bu_j^0=\textcolor{black}{\lp 1+\frac{(-1)^{j+1}\lceil j/2\rceil}{5}\epsilon\rp}\begin{pmatrix}\frac{y(10-y)}{25}\\ 0\end{pmatrix}\hspace{-0.8mm},\hspace{1mm}\text{and}\hspace{1mm} \bB_j^0=\begin{pmatrix}0\\0\end{pmatrix}.
\end{align*}
The \textcolor{black}{unperturbed ($\epsilon=0)$} initial velocity has a parabolic profile along the downstream direction and attains its maximum $\|\bu_{max}^0\|=1$ at $y=5$, \textcolor{black}{on the other hand, no magnetic field is assumed present initially.} \textcolor{black}{On the walls, we assign\begin{align*}
    \bu_j=\begin{pmatrix}0\\0\end{pmatrix}\hspace{-0.35mm},\hspace{1mm}\text{and}\hspace{1mm}\bB_j=\lp 1+\frac{(-1)^{j+1}\lceil j/2\rceil}{5}\epsilon\rp\begin{pmatrix}0\\1\end{pmatrix},
\end{align*}
for the velocity, and magnetic field, respectively, where the applied magnetic field is normal to the flow direction. As the inflow conditions, at the inlet we set \begin{align*}
    \bu_j=\lp 1+\frac{(-1)^{j+1}\lceil j/2\rceil}{5}\epsilon\rp\begin{pmatrix}\frac{y(10-y)}{25}\\0\end{pmatrix}\hspace{-0.35mm},\hspace{1mm}\text{and}\hspace{1mm}\bB_j=\lp 1+\frac{(-1)^{j+1}\lceil j/2\rceil}{5}\epsilon\rp\begin{pmatrix}0\\1\end{pmatrix}.
\end{align*}}\noindent For the outflow conditions, we extend the channel 10 units in the downstream direction and at the end we set outflow velocity and magnetic field equal to corresponding inflow conditions. \textcolor{black}{Thus, $\epsilon$ appears as a perturbation parameter in the initial and boundary conditions.} The initial and boundary conditions in the original variables are then transferred into the Els{\"{a}}sser variables. We generate a barycenter refined regular triangular unstructured mesh that provides a total of \textcolor{black}{$186,134$} velocity dofs, \textcolor{black}{$23,395$} pressure dofs, \textcolor{black}{$186,134$} magnetic field dofs, and \textcolor{black}{$23,395$} magnetic pressure dofs.

\textcolor{black}{For this computational} experiment, we consider a uniformly distributed random sample $\{(\nu_j,\nu_{m,j})\in [0.0009,0.0011]\times[0.009,0.011]\}$ with mean $(\Bar{\nu},\Bar{\nu}_m)=(0.001,0.01)$, \textcolor{black}{and no external source is considered in the system (i.e. $\bif_{j}=\bg_{j}=\textbf{0}$)}. We run the simulations using the Algorithm \ref{Algn1} with a fixed coupling parameter $s=0.001$, \textcolor{black}{varying the perturbation parameter $\epsilon$} and a fixed time-step size $\Delta t=0.05$  until an end time $T=40$. We plot the velocity and magnetic field ensemble average solutions in Figures \ref{channel-velocity}-\ref{channel-magnetic-field} \textcolor{black}{for various values of $\epsilon$. To make a comparison, we plot the solution for a single run simulation, which corresponds to the mean sample viscosities, and $\epsilon=0.0$ and present as `usual MHD' results. We observe that as $\epsilon\rightarrow 0$, the ensemble average solution converges to the usual MHD solution.} 

\begin{figure}[h!]
\begin{center}\vspace{-35mm}
            \includegraphics[width = 0.54\textwidth, height=0.4\textwidth,viewport=35 0 1400 890, clip]{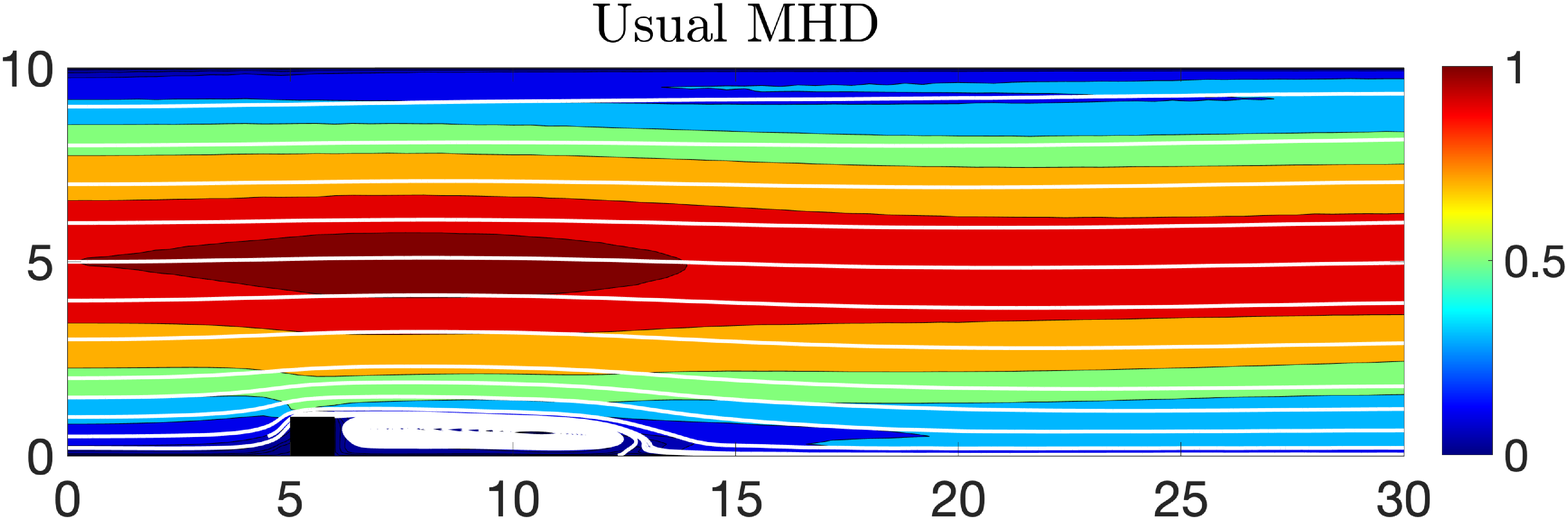}\vspace{-21.5ex}
            \includegraphics[width = 0.54\textwidth, height=0.4\textwidth,viewport=35 0 1400 890, clip]{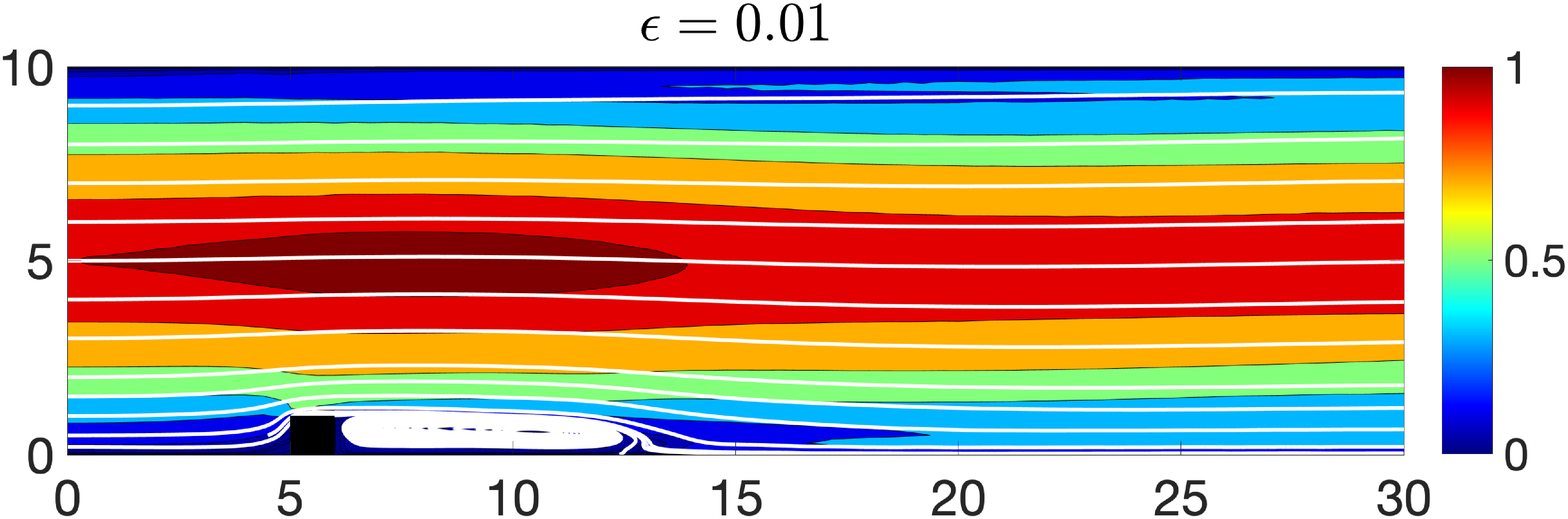}\vspace{-21.5ex}
            \includegraphics[width = 0.54\textwidth, height=0.4\textwidth,viewport=35 0 1400 890, clip]{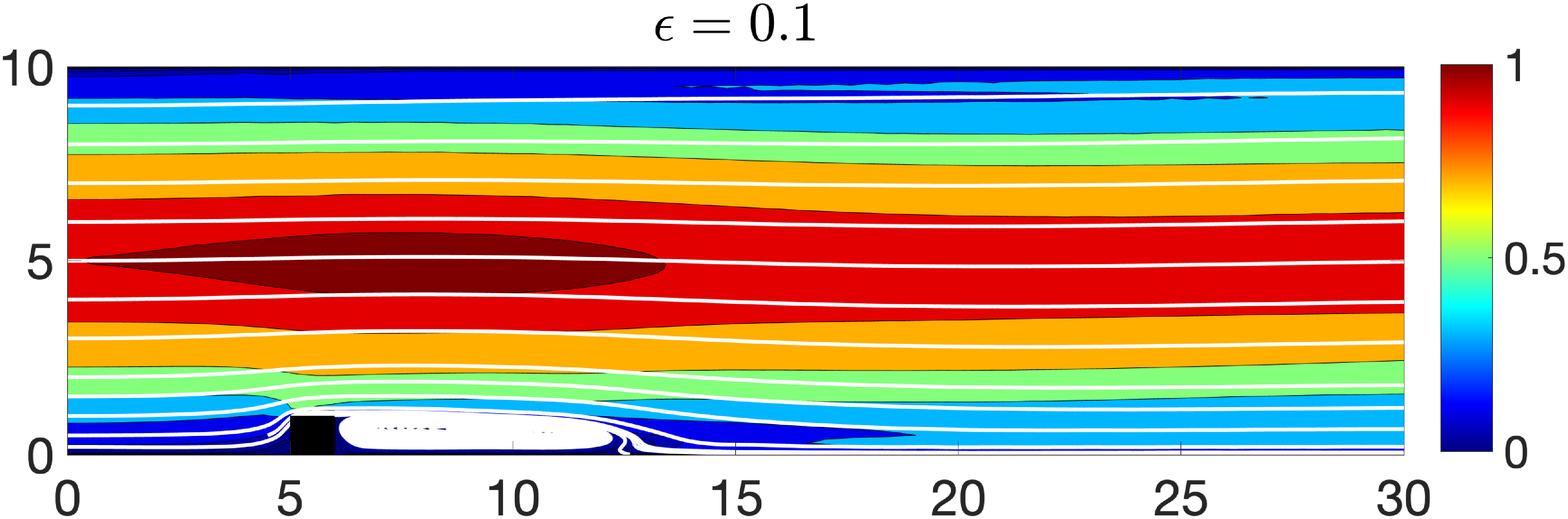}
	 	 	\caption{The velocity ensemble average solutions (shown as streamlines over speed contour) at $T=40$ for MHD channel flow over a step \textcolor{black}{for $s=0.001$}.}\label{channel-velocity}
\end{center}
\end{figure}

\begin{figure}[h!]
\begin{center}\vspace{-35mm}
            \includegraphics[width = 0.54\textwidth, height=0.4\textwidth,viewport=35 0 1400 890, clip]{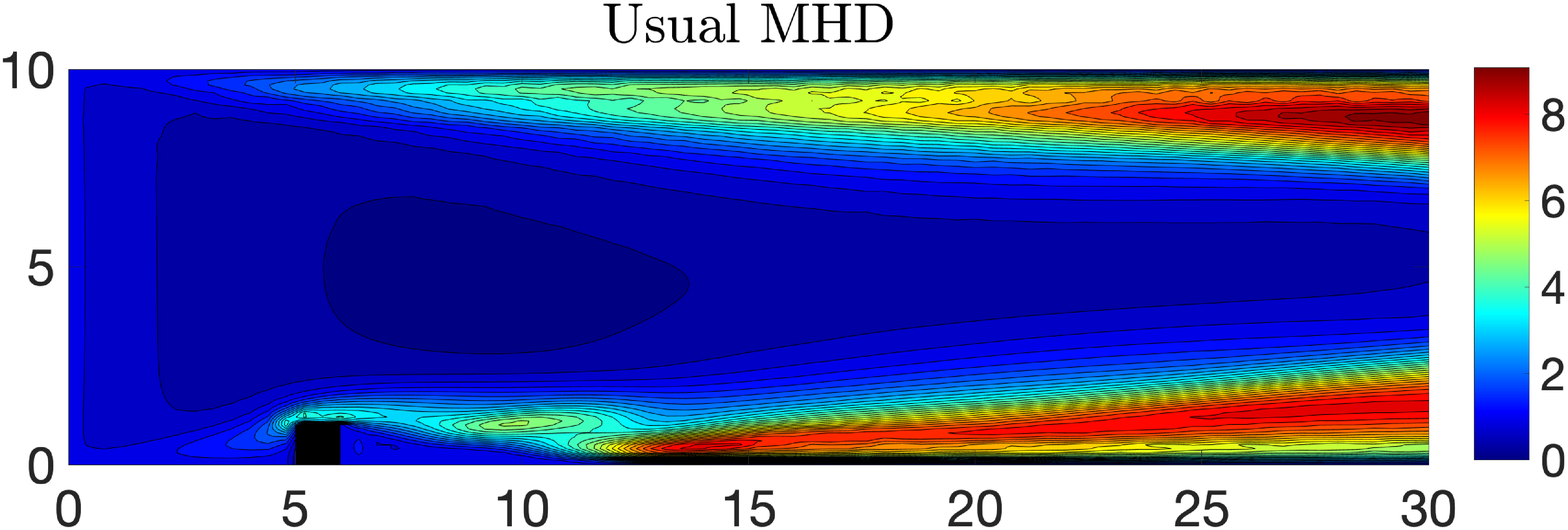}\vspace{-21.5ex}
            \includegraphics[width = 0.54\textwidth, height=0.4\textwidth,viewport=35 0 1400 890, clip]{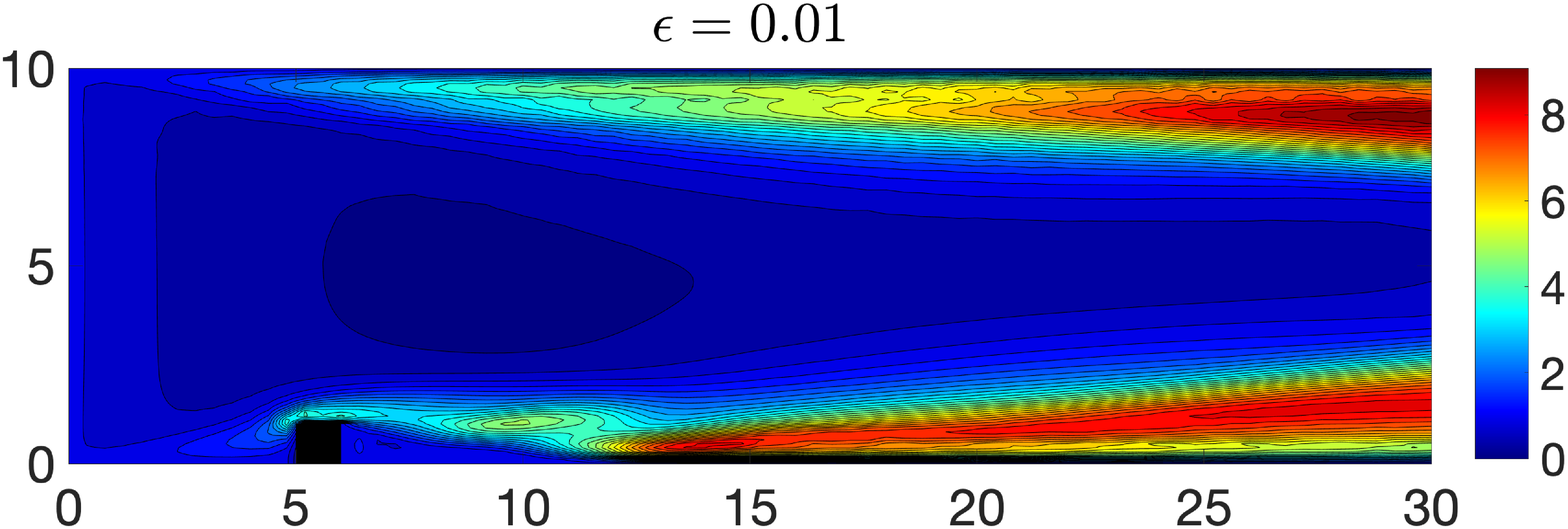}\vspace{-21.5ex}
            \includegraphics[width = 0.54\textwidth, height=0.4\textwidth,viewport=35 0 1400 890, clip]{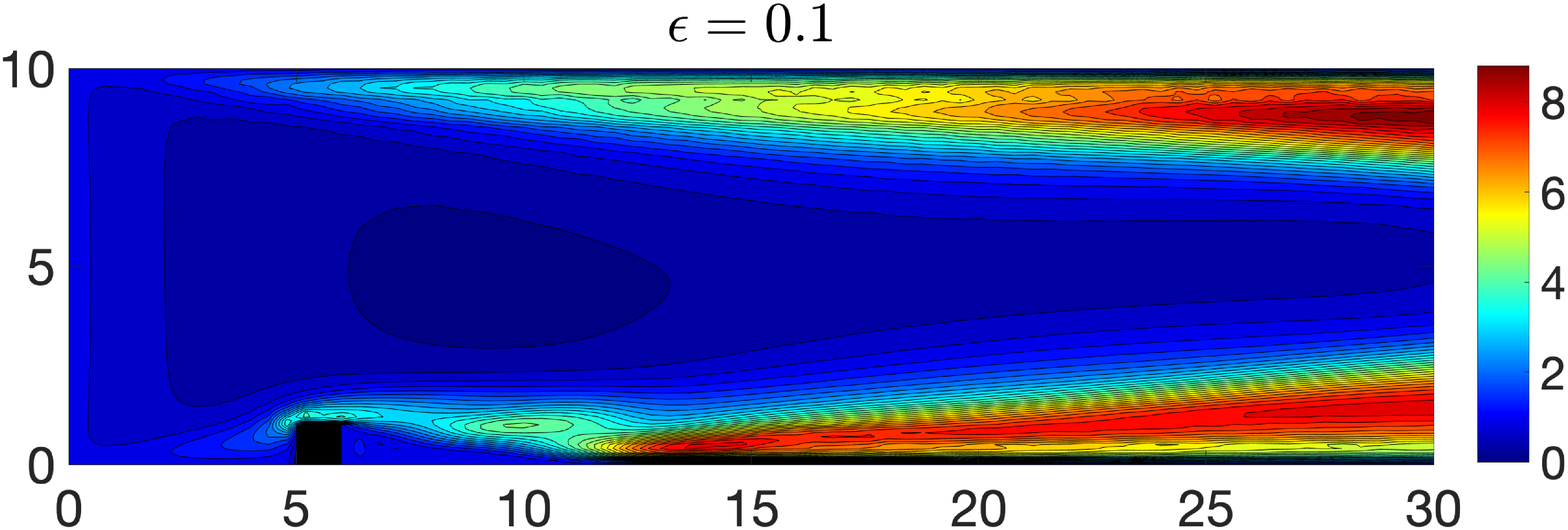}
	 	 	\caption{The magnetic field strength ensemble average solutions at $T=40$ for MHD channel flow over a step \textcolor{black}{for $s=0.001$}.}\label{channel-magnetic-field}
\end{center}
\end{figure}

\section{Conclusion and future works}\label{conclusion-future-works} In this paper,
we \textcolor{black}{have} proposed, analyzed, and tested an efficient ensemble algorithm, for a set of MHD simulations, which has the following features: (1) The linearized stable scheme is decoupled into two smaller identical subproblems, which can be solved at each time-step, simultaneously. This decoupling allows solving potentially much bigger problems with complex geometries than the MHD algorithms in terms of the primitive variables. (2) At each time-step, the system matrix remains common to all the ensemble members with different right-hand-side vectors. As a result, huge saving in storage and computational time, because the memory allocation for the system matrix, its assembly, factorization/preconditioners are needed only once per time-step. Moreover, the advantage of a block linear solver can be taken.  We assume the input data in the MHD flow involve uncertainties. Thus, each member \textcolor{black}{of the set is corresponding to a distinct combination of kinematic viscosity, magnetic diffusivity, initial conditions, boundary conditions, and body force.}

The unconditional stability of the scheme with respect to the time-step size is proven rigorously. The unconditional convergence is proven to be optimal in 2D, but in 3D the \textcolor{black}{theory} is suboptimal, due to the use of \textcolor{black}{the inverse inequality} in the analysis. \textcolor{black}{It is unclear if the suboptimal convergence is true, or if the 3D analysis is not sharp.} Numerical \textcolor{black}{experiments are performed} to verify the predicted convergence rates, \textcolor{black}{and energy stability of the scheme.} To observe the changes in the physical behavior as the coupling number increases \textcolor{black}{we have implemented the scheme} on a regularized lid-driven cavity with high Reynolds numbers. \textcolor{black}{We observe how solution changes as the deviation of noise in the initial and boundary conditions increases on} a channel flow past a rectangular step problems.

Our \textcolor{black}{future work} on MHD flow ensemble simulations will be based on N\'ed\'elec's edge element \cite{nedelec1980mixed} so that only the tangential component of the magnetic field becomes continuous across the inter-element boundaries. As a next step, this idea herein along with a penalty-projection \cite{AKMR15} ensemble algorithm for each subproblem can be \textcolor{black}{considered}. For high order  accurate uncertainty quantification along with the ideas proposed in \cite{Mohebujjaman2021High, wilson2015high} with deferred correction method will be the next research avenue. We will explore for more appropriate physical boundary conditions rather than the Dirichlet boundary conditions in Els\"asser variables. It has been shown in \cite{mohebujjaman2020scalability}, for Maxwell equations simulation, in presence of extremely different time scales, the iterative solver combination (FGMRES-GMRES) in conjunction with the parallel Auxiliary Space Maxwell (AMS) solver preconditioner outperforms over the direct solver. We plan to employ FGMRES-GMRES-AMS solver for solving complex problems using this scheme.

Parametric reduced order modeling (ROM) for MHD flow ensemble simulations following the data-driven approaches \cite{kaptanoglu2020physics,MRI18,xie2018data} and high order ROM differential filter in conjunction with evolve-filter-relax algorithm will also be the next research direction.

\bibliographystyle{plain}
\bibliography{High_order_MHD_ensemble}
\end{document}

%% file: notation-siam.tex
\newcommand{\lp}{\left(}
\newcommand{\rp}{\right)}

\newtheorem{remark}{Remark}[section]

\def\PP{{{\rm l}\kern - .15em {\rm P} }}
\def\PN2{{\PP_{N}-\PP_{N-2}}}

\newcommand{\cD}{\mathcal{D}}

\newcommand{\bfeta}{\boldsymbol{\eta}}

\newcommand{\bphi}{\boldsymbol{\varphi}}

\newcommand{\bpsi}{\boldsymbol{\psi}}

\newcommand{\ba}{\boldsymbol{a}}

\newcommand{\bb}{\boldsymbol{b}}
\newcommand{\bB}{\boldsymbol{B}}

\newcommand{\be}{\boldsymbol{e}}

\newcommand{\bH}{\boldsymbol{H}}
\newcommand{\bl}{\boldsymbol{l}}
\newcommand{\bL}{\boldsymbol{L}}

\newcommand{\bu}{\boldsymbol{u}}

\newcommand{\bv}{\boldsymbol{v}}
\newcommand{\bV}{\boldsymbol{V}}

\newcommand{\bw}{\boldsymbol{w}}

\newcommand{\bx}{\boldsymbol{x}}
\newcommand{\bX}{\boldsymbol{X}}

\newcommand{\bz}{\boldsymbol{z}}

\newcommand{\deleted}[1]{{}}
